\documentclass[natbib]{svjour3}   

\smartqed  
\usepackage{graphicx}
\usepackage{amssymb,amsmath,amscd}
\usepackage{bbm}
\usepackage{a4wide}
\usepackage{epsfig,psfrag}

\usepackage{chicago-bibstyle}  
\newcommand{\ts}{\hspace{0.5pt}}

\newcommand{\PP}{\mathbb{P}}

\newcommand{\NN}{\mathbb{N}}

\newcommand{\one}{\mathbbm{1}}

\newcommand{\cO}{\mathcal{O}}
\newcommand{\cL}{\mathcal{L}}

\newcommand{\cS}{\mathcal{S}}

\newcommand{\btimes}{\mbox{\huge \raisebox{-0.2ex}{$\times$}}}

\begin{document}

\title{Single-crossover recombination and ancestral recombination trees
}

\titlerunning{Ancestral recombination trees}        

\author{Ellen Baake         \and
        Ute von Wangenheim 
}


\institute{Ellen Baake \at
           Faculty of Technology \\
           Bielefeld University \\
           33594 Bielefeld \\
              Tel.: +49-521-1064896\\
              Fax: +49-521-1066411\\
              \email{ebaake@TechFak.Uni-Bielefeld.DE}           
           \and
           Ute von Wangenheim \at
             Faculty of Technology \\ Bielefeld University \\
             33594 Bielefeld \\
             \email{uvonwang@TechFak.Uni-Bielefeld.DE}  
}

\date{Received: date / Accepted: date}

\maketitle

\begin{abstract}
We consider the Wright-Fisher model for a population of $N$ individuals,
each identified with a sequence of a finite number of  sites, and
single-crossover recombination between them. We trace back the ancestry of
single individuals from the present population. In the $N \to \infty$ limit
without rescaling of parameters or time,
this ancestral process is described by a  random tree, whose
branching events correspond to the splitting of the sequence due to
recombination. With the help of a decomposition of the trees into subtrees,
 we calculate the probabilities of the topologies of the ancestral trees. At the
same time, these probabilities lead to a semi-explicit solution of
the deterministic single-crossover equation. The latter is
a discrete-time dynamical system that emerges from
the Wright-Fisher model via a law of large numbers 
and has been waiting for a solution for many decades.
\end{abstract}

\section{Introduction}
\label{intro}
Recombination happens
during sexual reproduction and refers to the combination of the genetic
material of two parents into the `mixed' type of an offspring individual.
More precisely, the recombined offspring results from
a reciprocal exchange of maternal and paternal gene sequences 
via so-called {\em crossovers}.
Due to the interaction of individuals and due to dependencies between
the positions at which recombination may take place, the
process is difficult to handle. This applies even to the limit of infinite
population size, where a law of large numbers turns
the dynamics of gene frequencies into a deterministic, nonlinear
system of difference or differential equations, which has challenged
population geneticists since its first formulation by Geiringer
in 1944; see 
\cite{Bennett, HaleRingwood, Kevin1, Kevin2}; and \cite{reco} for a sample of subsequent work.
Under the so-called single-crossover assumption, 
where at most one crossover occurs in any gene sequence in every generation,
the deterministic model can be solved explicitly (and in an
astonishingly simple way) in \emph{continuous} time \citep{reco}. But the corresponding
discrete-time dynamics, which is prevalent in the biological
literature, is more difficult; its solution has, so far,
required nontrivial transformations and recursions that have not yet been
solved in closed form  \citep{Bennett, Kevin1, Kevin2, discretereco}.

In this paper, we will present a \emph{semi-explicit} solution to the
\emph{discrete-time single-crossover population model} by considering
the \emph{ancestry of single individuals}. The original
\emph{deterministic forward-time dynamics} is thus considered in terms of
a \emph{stochastic process backward in time}, whose solution leads 
 to that of the original system.  In the backward process, one gains
a certain conditional independence of  gene segments, which will allow for a
solution. In this sense, a probabilistic
representation provides the necessary understanding to solve the original
deterministic problem. 

More precisely, we proceed as follows.
In Section~\ref{sec:WFmodel}, we start from the stochastic (i.e., finite-population)
version of the discrete-time single-crossover model, that is, the
Wright-Fisher model with single-crossover recombination 
\cite[Chap. 5.4]{HeinSchierupWiuf}. In the limit
of population size tending to infinity (without rescaling of parameters or time),
 a law of large numbers (established here
explicitly) leads to the corresponding deterministic dynamical
system. We recall some general properties of this system and
discuss the various dependencies (between individuals and between
gene segments) that have, so far, obstructed an explicit solution.
In Section~\ref{sec:ancestral}, we take the backward point of view and consider the ancestry
of the genetic material of single individuals. In the
limit of infinite population size, this ancestry is a random tree
for any finite time horizon,
that is, segments that have been separated once do not come together again in
the same individual (with probability one).
The law for this ancestral tree may be formulated explicitly in terms of
a (stochastic) segmentation process, which involves conditional
independence between segments once they appear. As a consequence,
the time evolution of the ancestral process may be calculated 
via a decomposition into subtrees. This solution is semi-explicit in
the sense that it is a sum of well-defined terms, where summation is
over certain tree topologies, which must be enumerated in a recursive way.
In the same sense, this yields a  solution of the deterministic
forward-in-time model. We will discuss our results in the context
of related approaches in Section~4,
in particular, the ancestral recombination graph (the usual approach to
recombination in finite populations).

\section{The recombination model forward in time}\label{sec:WFmodel}
\subsection{The model}
In this section, we  describe the basic setting, the 
\emph{Wright-Fisher model
with single--crossover recombination},
as well as the dynamical system (from \citealt{discretereco}) that arises
as its infinite-population limit. 
A chromosome is described by a linear arrangement of, say, $n+1$ {\em sites},
namely, the elements of the set $S:=\lbrace 0, 1 , \ldots , n \rbrace$.
Sites represent discrete positions on a chromosome that may be interpreted
as gene or nucleotide positions. Thus, each site $i\in S$ can be occupied by
an {\em allele} (or \emph{letter}) $x^{}_i\in X^{}_i$, where 
we restrict ourselves 
 to {\em finite} $X^{}_i$.
A {\em type} $x$ is then defined as a sequence 
$x = (x^{}_0, x^{}_1 ,\ldots ,x^{}_n)\in X^{}_0\times X^{}_1\times\cdots\times X^{}_n =:X$,
where $X$ denotes the (finite) \emph{type space}.
Neighbouring sites are connected by {\em links}, the entities where recombination events may occur.
They are collected into the set 
$L=\lbrace{\frac{1}{2}, \frac{3}{2}, \ldots ,\frac{2n-1}{2}\rbrace}$, where
link $\alpha=\frac{2i+1}{2}$ denotes the link between sites $i$ and $i+1$.
We will only be concerned with single crossovers, i.e., the case where recombination occurs at
a single link $\alpha\in L$ and results in 
a mixed type composed of the sites before $\alpha$ from the first parent,
and those after $\alpha$ from the second parent.
Explicitly, if recombination involves the ordered pair of types
$x = (x^{}_0, \ldots, x^{}_n)$ and $y=(y^{}_0, \ldots, y^{}_n)$, the outcome 
of recombination at  link $\frac{2i+1}{2}$
is the recombined type $(x^{}_0, \ldots, x^{}_i, y^{}_{i+1}, \ldots, y^{}_n)$. 
The dynamics of a finite population that evolves under single-crossover recombination
can be described by the following version of the Wright-Fisher model (cf.\ 
\citealt[Chap. 5.4]{HeinSchierupWiuf}): 

\label{wf-scr}
Each link is equipped with a crossover probability $\varrho_{\alpha}>0$
(with $\sum_{\alpha \in L} \varrho_{\alpha} \leqslant 1$).
Each generation is  of constant size $N$. 
In each generation, the current population is replaced by its offspring,
where each offspring individual chooses
its parent(s) independently  according to the following scheme
(see Figure~\ref{fig:wright}):
\begin{itemize}
\item With probability $\varrho^{}_{\alpha}>0$, $\alpha\in L$,
two parents are chosen uniformly with replacement. They recombine at link 
$\alpha$, which gives rise
to the corresponding recombined offspring with the leading segment 
(the sites $0, \ldots, \lfloor \alpha\rfloor$) from the first and
the trailing segment (the sites $\lceil\alpha\rceil,\ldots, n$)
from the second parent, where $\lfloor\alpha\rfloor$ ($\lceil\alpha\rceil$) 
denotes the largest integer below (the smallest above) $\alpha$; 
if the same parent is chosen
twice, it is effectively transmitted unchanged.
 \item With probability $0\leqslant 1-\sum_{\alpha\in L}\varrho^{}_{\alpha}<1$, a single parent
is selected uniformly and with replacement from the previous generation.
\end{itemize}

\begin{figure}
\begin{center} 
  \includegraphics[width=10cm,height=4.5cm]{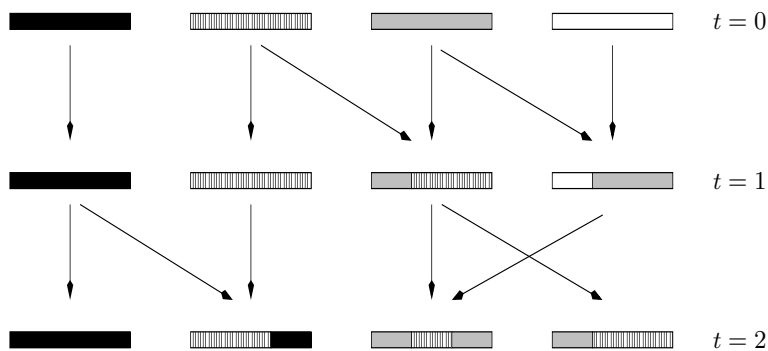}
  \end{center}  
\caption{Wright-Fisher model with single--crossover recombination  for $N=4$.}
\label{fig:wright}
\end{figure}
We denote the population at time $t$ 
by 
\[
Z^{}_t = (Z^{}_t(x))_{x\in X}\in E:=\{ \nu \text{ counting measure on }
X \mid \lVert \nu \rVert =N  \},
\]
where $\lVert . \rVert$ denotes total variation norm
and $Z^{}_t(x)$ is the number of individuals of type $x$ at time $t$.
We will also need the corresponding normalised quantity
$\widehat{Z}_t :=Z^{}_t/N$, which is a probability vector (or measure)
for the  population at time $t$.
In order to formalise the stochastic process,
let $\pi^{}_J \! : X \longrightarrow X^{}_J := \btimes_{i\in J}X^{}_i$,
$\pi^{}_J(x) =(x^{}_i)^{}_{i\in J} =:x^{}_J$,
be the canonical projection to the sites in $J$ ($J\subseteq S$).
We specifically need 
$\pi^{}_{<\alpha} :=   \pi^{}_{\lbrace0,\ldots,\left\lfloor \alpha\right\rfloor\rbrace}$ and
$\pi^{}_{>\alpha} := \pi^{}_{\lbrace\left\lceil \alpha\right\rceil,\ldots,n\rbrace}$.
For $p\in \mathcal{P}(X)$, the set of probability measures on $X$,
we denote by $\pi^{}_{J\cdot}p := p\circ\pi^{-1}_{J}$
(where $\pi^{-1}_{J}$ 
denotes the preimage of $\pi^{}_{J}$) 
the marginal distribution of $p$ with respect to the sites in $J$.
Furthermore, 
\begin{equation}\label{eq:recombinator}
  R_{\alpha}(p) := (\pi^{}_{<\alpha\cdot}p)\otimes(\pi^{}_{>\alpha\cdot} p) 
\end{equation}
is the product measure of the two marginals (before and after $\alpha$);
$R^{}_{\alpha}$ is known as the recombination operator (or \emph{recombinator} for short), cf. \cite{reco}. 
It is clear that an individual that recombines at link $\alpha\in L$ in generation $t$ draws its
type from $R^{}_{\alpha}(\widehat{Z}^{}_{t-1})$, and
a non-recombining individual draws its type from 
 $\widehat{Z}^{}_{t-1} =R^{}_{\varnothing}(\widehat{Z}^{}_{t-1})$,
with $R^{}_{\varnothing} := \one$ (the reason for this notation will
become clear later).

The discrete-time Markov chain $\{\widehat{Z}^{}_t\}_{t\in\mathbb{N}_0}$  on $\mathcal{P}(X)$
may therefore be formulated as follows:
\begin{itemize}
\item Let $N^{}_{\alpha}(t)$, $\alpha\in L$, 
denote the random number of  individuals
generated in generation $t$ via recombination at link $\alpha\in L$.
Analogously, $N^{}_{\varnothing}(t)$ is the number of individuals that are
sampled without recombining. Clearly, they follow a multinomial distribution:
\begin{equation}\label{stoch1}
 (N^{}_{\varnothing}(t),N^{}_{\frac{1}{2}}(t),\ldots, N^{}_{\frac{2n-1}{2}}(t))
 \sim \mathcal{M}(N,(1-\sum_{\alpha\in L}\varrho^{}_{\alpha}, \varrho^{}_{\frac{1}{2}},\ldots,\varrho^{}_{\frac{2n-1}{2}})) \, ,
 \quad \text{i.i.d for all } t \, .
\end{equation}
\item 
According to the previous step, $Z_t$ consists of subpopulations
$Y_{\beta}(t)$, 
$\beta \in L \cup \{\varnothing\}$, 
where $Y^{}_{\beta}(t)$ consists   of those individuals that, in generation $t$,
experience recombination at
$\beta$ (where  $\beta=\varnothing$ indicates no recombination).
Clearly,
\begin{equation}\label{stoch2}
 Y^{}_{\beta}(t) \sim \mathcal{M}(N_{\beta}(t),R^{}_{\beta}(\widehat{Z}_{t-1})), \quad \beta\in L\cup \{\varnothing\}.
\end{equation}
\item Finally, we obtain $\widehat{Z}^{}_{t}$
via 
\begin{equation}\label{stoch3}
 \widehat{Z}^{}_{t} = \frac{1}{N}(Y^{}_{\varnothing}(t) + \sum_{\alpha\in L} Y^{}_{\alpha}(t)) \, .
\end{equation}
\end{itemize}

Obviously, the resampling-recombination mechanism is  independent of the 
types. So, the Wright-Fisher model   
may, alternatively, be constructed as an independent superposition 
of the two processes, that is,
\begin{enumerate}
\item[{\bf (F1)}] It is first determined, for each time point
and for each individual,
which of the sites come from which parental individual 
(resampling/recombination without types).
\item[{\bf (F2)}]
Letters are then attached to the sites at time $t=0$ 
and are then propagated through the model to time $t$
according to the relations 
decided in {\bf (F1)}.
\end{enumerate}

\subsection{Law of large numbers.}
Let us first consider the Wright-Fisher model in the so-called
{\em infinite population limit (IPL)}, where we let $N\rightarrow\infty$ 
without rescaling any other parameters.
This may be considered as a limit of \emph{strong recombination},
in which the stochastic effects of resampling (also known as genetic
drift) are lost. This is in contrast to
 the more frequently used weak recombination
limit, which leads to a diffusion process, compare \citet[Chap.~6.6]{Ewens}
and Section~\ref{sec:discussion} below.

More precisely, we consider the family of processes
$\{\widehat{Z}^{(N)}_t\}_{t\in\mathbb{N}_0}$, $N \in \NN$
(where we temporarily add an upper index $N$
to denote population size)
and compare it with the deterministic recombination dynamics,
where we identify
the  population at time $t\in \mathbb{N}_0$  with
$p^{}_t= (p_t(x))_{x\in X}\in\mathcal{P}(X)$.
Here $p^{}_t(x)$ denotes 
the relative frequency of type $x\in X$ at time $t$, and 
$p^{}_0$ is the initial population.
The population is described by the dynamical system
\begin{equation} \label{rekooper}
    p^{}_{t}  =  \varPhi(p_{t-1}), \quad \text{where } \,
      \varPhi(p) :=  \Big (1-\sum_{\alpha\in L}\varrho^{}_{\alpha} \Big ) p + 
   \sum_{\alpha\in L} \varrho^{}_{\alpha}  R_{\alpha}(p) \, ,
\end{equation}
which is usually obtained by direct deterministic modelling \citep{discretereco}.
The following result shows that, indeed, \eqref{rekooper} describes the  infinite population limit
of the stochastic process, more precisely for the family of processes
$\{\widehat{Z}^{(N)}_t\}_{t\in\mathbb{N}_0}$, $N \in \NN$. We will
use $\varPhi^{t+1}=\varPhi \circ \varPhi^t$ for the composition
of the nonlinear mapping $\varPhi$.

\begin{proposition}[Infinite Population Limit]\label{prop:IPL}
Let $\{\widehat{Z}^{(N)}_t\}_{t\in\mathbb{N}_0}$ with $N \in \NN$ be a family 
of Wright-Fisher models with single-crossover recombination
(as  defined by
\eqref{stoch1}--\eqref{stoch3}) with initial states such that
$\lim_{N\rightarrow\infty} \widehat{Z}_0^{(N)} = p^{}_0$. Then, for every given $t\in\mathbb{N}_0$, 
one has
\begin{equation}\label{IPLconv}
\lim_{N\rightarrow\infty} \widehat{Z}_t^{(N)} = p^{}_t \quad
\text{in mean square},
\end{equation}
where $p^{}_t =\varPhi^t(p^{}_0)$ denotes the  solution of \eqref{rekooper}.
\end{proposition}
The corresponding  situation in \emph{continuous} time and with 
\emph{almost sure convergence} (see Remark \ref{rem:interval})
is covered by the  general law of large numbers
of  \citet[Theorem 11.2.1]{EK}, but no such general 
result seems to be available in discrete time.
We therefore include a proof.
\begin{proof}[of Prop. \ref{prop:IPL}]
We employ induction over $t$.
By assumption, the claim holds for $t=0$. Now assume that
it holds for $t-1$, for some $t \geqslant 1$. 
We then have
\begin{equation}\label{facta}
 \frac{N_{\alpha}^{(N)}(t)}{N} \xrightarrow{ N \to \infty} \varrho^{}_{\alpha} \,, \alpha\in L, \quad \text{and}\quad
 \frac{N_{\varnothing}^{(N)}(t)}{N} \xrightarrow{ N \to \infty} 1- \sum_{\alpha\in L} \varrho^{}_{\alpha}  
\end{equation}
by the mean square law of large numbers
(cf.\ \citealt[Chap.~7.4]{Grimmett}).
Furthermore, for $\beta\in L$, $N_{\beta}^{(N)}(t)\rightarrow \infty$ 
as $N\rightarrow \infty$ with probability one (because $\varrho^{}_{\beta}>0$) 
and thus 
\begin{equation}\label{factb}
 \frac{Y_{\beta}^{(N)}(t)}{N_{\beta}^{(N)}(t)} 
  \xrightarrow{ N \to \infty} R^{}_{\beta}(p_{t-1}) \quad \text{in mean square,}
\end{equation}
since 
$Y_{\beta}^{(N)}(t)/N_{\beta}^{(N)}(t) - R^{}_{\beta}(\widehat{Z}^{(N)}_{t-1})
\xrightarrow{ N \to \infty} 0 $
due to the mean-square law of large numbers (except on the set of measure 0
where $N_{\beta}^{(N)}(t)\nrightarrow \infty$, and
thus altogether in mean square), and 
$R^{}_{\beta}(\widehat{Z}^{(N)}_{t-1})  \xrightarrow{ N \to \infty}
R^{}_{\beta}(p_{t-1})$ by the induction hypothesis.
Analogously, for $1-\sum_{\alpha\in L}\varrho^{}_{\alpha}>0$,
\begin{equation}\label{factc}
 \frac{Y^{(N)}_{\varnothing}(t)}{N^{(N)}_{\varnothing}(t)}
\xrightarrow{ N \to \infty} p_{t-1}
\quad\text{ in mean square}.
\end{equation}
Since, by \eqref{stoch3},
\[
 \widehat{Z}^{(N)}_{t} = 
\frac{N^{(N)}_{\varnothing}(t)}{N}\cdot \frac{Y^{(N)}_{\varnothing}(t)}{N^{(N)}_{\varnothing}(t)}
  + \sum_{\alpha\in L} \frac{N^{(N)}_{\alpha}(t)}{N} \cdot \frac{Y^{(N)}_{\alpha}(t)}{N^{(N)}_{\alpha}(t)},
\]
\eqref{facta}--\eqref{factc} together tell us that
\begin{equation}\label{factconseq}
  \widehat{Z}^{(N)}_{t}  \xrightarrow{ N \to \infty}
  \varPhi(p_{t-1})
 \quad \text{in mean square},
\end{equation}
which proves the claim. \qed
\end{proof}

\begin{remark}\label{rem:interval}
Note  that Prop.~\ref{prop:IPL} automatically implies  that, for every
given $t$,
\[
\lim_{N\rightarrow\infty} \max_{s\leqslant t}\lvert 
\widehat{Z}_s^{(N)} - p^{}_s \rvert =0  \quad
\text{in mean square},
\]
which is reminiscent of the continuous-time result  \cite[Theorem 11.2.1]{EK}.
Note, however, that the latter result is a \emph{strong} law of large
numbers; we have established the mean-square version here 
(which, of course, implies a weak law of large numbers since
convergence in mean square implies convergence in probability) since the
construction of a sequence of processes on the \emph{same} probability space
in the discrete-time setting is beyond the scope of this paper.
Note also that the convergence in \eqref{IPLconv} applies for 
any fixed $t\in\mathbb{N}_0$,
but need {\em not}  hold as $t\rightarrow\infty$. 
Indeed, the asymptotic behaviour of the stochastic system is radically
different from that of the deterministic one:
Due to resampling, the Markov chain is absorbing
(in fact, it experiences fixation of a single type
with probability one in the long run). In contrast, the deterministic system
never loses any type, and the complete product measure 
with respect to all links in $L$ 
is obtained as the stationary distribution, see \cite{Geiringer} and
\cite{discretereco}.
Let us emphasise  that it is the short time scale, not the 
long-term behaviour, that we are interested in here; see Section~\ref{sec:discussion}
for the discussion of the biological context.

\end{remark}
\subsection{Structure of the deterministic solution}
Based upon an initial population $p^{}_0$,  every individual in the population at time $t=1$
is either an unaltered copy of an individual from $p^{}_0$ or it is composed of
exactly two recombined segments, hence the population $p^{}_1$
is a mixture of $p^{}_0$ and the  $R^{}_{\alpha}(p^{}_0)$, $\alpha\in L$, in line with \eqref{rekooper}.
For $t>1$, the population will contain
individuals that consist of several segments pieced together 
from the sequences in the initial population  due to various
recombination events at different times.  To describe these, we use the
{\em composite} recombinators $R^{}_{G}$, $G\subseteq L$, which
act on probability vectors as 
\begin{equation} \label{eq:defcomposite}
    R^{}_{G} \, := \, \prod_{\alpha\in G} R^{}_{\alpha},
\end{equation}
where we set $R^{}_{ \{ \alpha\} } = R^{}_{\alpha}$. 
Here, the product is to be read as composition. It is, indeed, a 
matrix product if the recombinators
are written in their matrix representation, which is available in 
the case of finite
types considered here, provided the problem is embedded into
a larger space \citep{ellen1}. 
This definition is consistent since all $R^{}_{\alpha}$ are idempotents and commute with
each other, compare \cite{reco}.
Clearly, $R^{}_G(p)$ is the product measure derived from $p$ with respect to all links in $G$.
We thus expect the population at any time to be a 
convex combination of the  $R^{}_G(p^{}_0)$ with $G\subseteq L$.
This means
\begin{equation}\label{ansatzsol}
  p^{}_t \ts  = \varPhi^t(p^{}_0)\ts = \sum_{G\subseteq L} a^{}_G ( t ) R^{}_G ( p^{} _0 ) \, ,
\end{equation}
with $a^{}_G(0) = \delta^{}_{G,\varnothing}$,
 $a^{}_G(t)\geqslant 0$ for all $G\subseteq L$, and $\sum_{G\subseteq L}a^{}_G(t)=1$.
It has been proved by \cite{discretereco} that the solution indeed has this form,
but plausibility arguments go back to \cite{Geiringer}.
The difficulty consists in determining the coefficient functions
$a^{}_G(t)$.
Let us introduce the following abbreviations,
\begin{equation*}
    \begin{split}
    G^{}_{< \alpha} &:= \left \{ \beta \in G \mid \beta < \alpha \right \} , \quad
    G^{}_{> \alpha}  :=  \left \{ \beta \in G \mid \beta > \alpha \right \} ,\\
    G^{}_{ \leqslant \alpha} &:= \left \{ \beta \in G \mid \beta \leqslant \alpha \right \} , \quad
    G^{}_{ \geqslant \alpha}  := \left \{ \beta \in G \mid \beta \geqslant \alpha \right \}. 
  \end{split}
\end{equation*} 

Let us recall the recursion for the coefficient functions from
\cite{discretereco}:
\begin{theorem}\label{thm:adevelop}
  For all $G\subseteq L$ and $t\in\mathbb{N}_0$, the coefficient functions $a^{}_G(t)$ evolve according to
   \begin{equation}\label{nonlinrecur}
     a^{}_G(t+1)\, = \,  \Bigl( 1 - \sum_{\alpha\in L} \varrho^{}_\alpha \Bigr) 
     \thinspace a^{}_G(t) + \sum_{\alpha \in G } \varrho^{}_{ \alpha} 
          \Bigl( \sum_{H \subseteq L_{ \geqslant \alpha} }  a^{}_{G^{}_{ < \alpha}\cup H}(t) \Bigr) \thinspace
          \Bigl( \sum_{K \subseteq L_{ \leqslant \alpha }} a^{}_{K \cup G^{}_{> \alpha} }(t) \Bigr) \ts ,
   \end{equation}
with initial condition $a^{}_G(0) = \delta^{}_{G,\varnothing}$.
\qed
\end{theorem}
A verbal description of this iteration can already be found in \cite{Geiringer}.
It will become clear later that we may interpret 
$a^{}_G(t)$ as the proportion of the population whose types
have been pieced together by recombination at 
{\em exactly} the links of $G$.

Due to its nonlinearity, the recursion
does not allow
for an immediate solution (at least from four sites onwards).
The nonlinearity comes from the \emph{dependence} of links:
Due to the single-crossover assumption, a crossover event 
forbids any other recombination events in the same time step.
In sharp contrast, and quite surprisingly,
the analogous (deterministic) single--crossover  model in
{\em continuous} time has a very simple explicit solution
\citep{reco,MB}.
The main reason for this is 
the fact
that simultaneous crossover events are \emph{automatically}
excluded in continuous time.
This implies an  {\em effective independence} of links, which
turns the dynamics corresponding to Theorem~\ref{thm:adevelop} into a
\emph{linear} one.
For a detailed investigation of the differences between single-crossover 
dynamics in continuous and in discrete time, 
the reader is  referred to \cite{discretereco}.

The conventional way \citep{Bennett,Kevin1,Kevin2}
to overcome the obstacles of nonlinearity in recombination models
lies in finding an appropriate transformation of the dynamics to
a solvable diagonalised system, 
but this usually  involves a new set of coefficients that must be constructed
in a recursive manner.
We have performed this for the single-crossover model
\citep{discretereco},
 but the solution still requires recursions
and does not lead to closed-form expressions for the $a^{}_G(t)$.
In contrast, we will pursue the stochastic perspective here and look at
recombination \emph{backward} in time, which will lead us to
the coefficient functions in semi-explicit form.

\section{Ancestral recombination process}\label{sec:ancestral}
\subsection{The ancestral process}
In the {\em ancestral recombination process},
we follow the ancestry of the genetic material of
a selected individual from a population that evolved
according to the Wright-Fisher model with single-crossover recombination
 of Section~\ref{wf-scr}.
To this end, we start with an individual in the present population at time $t$
and let time run backwards, as illustrated in Figure~\ref{fig:wfances}
for two individuals from the realisation of the Wright-Fisher
model in Figure~\ref{fig:wright}. Let us first describe the resulting partitioning
of sites into  parents, keeping in mind that 
this happens independently of the types, in analogy with 
step {\bf (F1)} in the forward model. 

We denote by  $\tau$ the time backward from the present at time $t$, 
i.e., backward
time $\tau$ corresponds to forward time $t-\tau$. 
(Note that, altogether, we use the symbol $t$  both for the variable 
of time and for the
fixed number of generations for which the (forward-time) dynamics is
considered. In the latter sense, $t$ stands for `today'.)
We capture the
partitioning by a process $\{\Sigma_\tau\}_{\tau \in \NN_0}$ on 
$\Pi(S)$, the set of 
partitions of $S$. Here, the parts of $\Sigma_\tau$ correspond to
the  parents at (backward) time $\tau$ of our individual at (forward)
time $t$; sites in the
same part correspond to sites that go back to the same parent. 
In view of the forward Wright-Fisher model, it is clear that
$\{\Sigma_\tau\}_{\tau\in \NN_0}$ is declared as follows.

Start with $\Sigma_0 = \{ S\}$. 
Assume now that for some $\tau > 0$, 
$\Sigma_\tau = \sigma := \{ \sigma^{}_1,  \dots , \sigma^{}_k\}$, where
$\sigma^{}_j=\{ \sigma^{}_{j1} , \dots , \sigma^{}_{jn_j}\}$ and we imply 
$\sigma^{}_{j1} < \sigma^{}_{j2} < \dots <
\sigma^{}_{jn_j}$, $1 \leqslant j \leqslant k$. Referring back to
the Wright-Fisher model, $\Sigma_{\tau+1}$ 
is obtained in two steps:
\begin{enumerate}
\item[(S)] Splitting: Every part $\sigma_j$ of 
$\Sigma_\tau$,
$1\leqslant j \leqslant k$, independently of the others, either remains unchanged 
(probability $1-\sum_{\sigma_{j1} < \alpha < \sigma_{jn_j}} \varrho_\alpha$), 
or, for 
every $\sigma^{}_{j1}<\alpha < \sigma^{}_{jn_j}$, it may split into 
$\{ \sigma^{}_{j1}, \dots, 
\sigma^{}_{j\lfloor\alpha\rfloor}\}$ and 
$\{ \sigma^{}_{j\lceil\alpha\rceil},\ldots, \sigma^{}_{jn_j}\}$  
(probability $\varrho_\alpha$). Note that two or more $\alpha$'s can lead to
the same split if $\sigma^{}_j$ is not contiguous,  
where `contiguous' means an uninterrupted run of sites.
The resulting refined partition is denoted by  $\Sigma_{\tau}'$.
This step corresponds to the splitting 
of the ancestral material into smaller segments due to recombination,
where we do not yet decide which segment ends up in which parent.
\item[(C)] Coalescence:  
Each 
part of  $\Sigma_{\tau}'$ now chooses 
one out of $N$  parents, uniformly and with replacement. Parts 
that end up in the same parent are united; otherwise, nothing happens.
The resulting partition is 
$\Sigma_{\tau + 1}$. Figure~\ref{fig:wfances} illustrates this:
If  all parts are assigned to different parents,
then no coalescence takes place, that is, $\Sigma_{\tau + 1} = \Sigma_{\tau}'$
(as in Figure~\ref{fig:wfances}, left). 
If two or more parts go back to the same parent, we have
a  coalescence event,
see Figure~\ref{fig:wfances} (right). 

\end{enumerate}
A closely related process describing the ancestry of single individuals
in \emph{continuous time} and on a \emph{continuous chromosome}
was investigated by 
\cite{WiufHein}, but in the weak-recombination limit, and with a
different purpose; we will come back to this in Section~\ref{sec:discussion}. 

\begin{figure}
\begin{center} 
  \includegraphics[width=8cm,height=3cm]{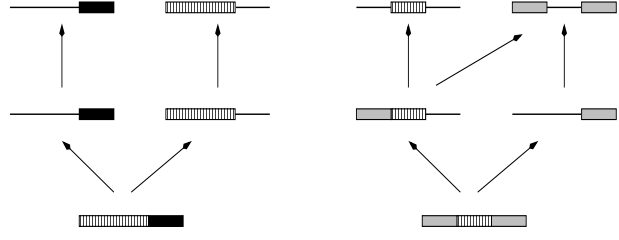}
  \end{center}  
\caption{Ancestries for two individuals from the Wright-Fisher population
of Figure~\ref{fig:wright}.
We trace back the ancestry of the segments present at $t=2$; 
the thin black lines indicate nonancestral material whose history is not relevant.
The left graph refers to the second individual from Figure \ref{fig:wright};
here the two segments go back to two different ancestors. The right graph corresponds to the third
individual from Figure \ref{fig:wright}. Two of its three segments 
have the same parent at $t=0$ due to a coalescence event. 
In the infinite population limit, such a situation  does not occur; rather, all
possible ancestries are binary trees.
 }
\label{fig:wfances}
\end{figure}

Our aim is now to determine the law for the ancestry and the type of 
a random individual at time $t$ \textit{without} constructing a realisation 
of the forward Wright-Fisher model first. 
Such an individual, together with its ancestry, may be 
constructed in a three-step procedure, see Figure~\ref{fig:bigtree}.
\begin{figure}
\begin{center}\hspace{-2mm} 
  \includegraphics[scale=0.3]{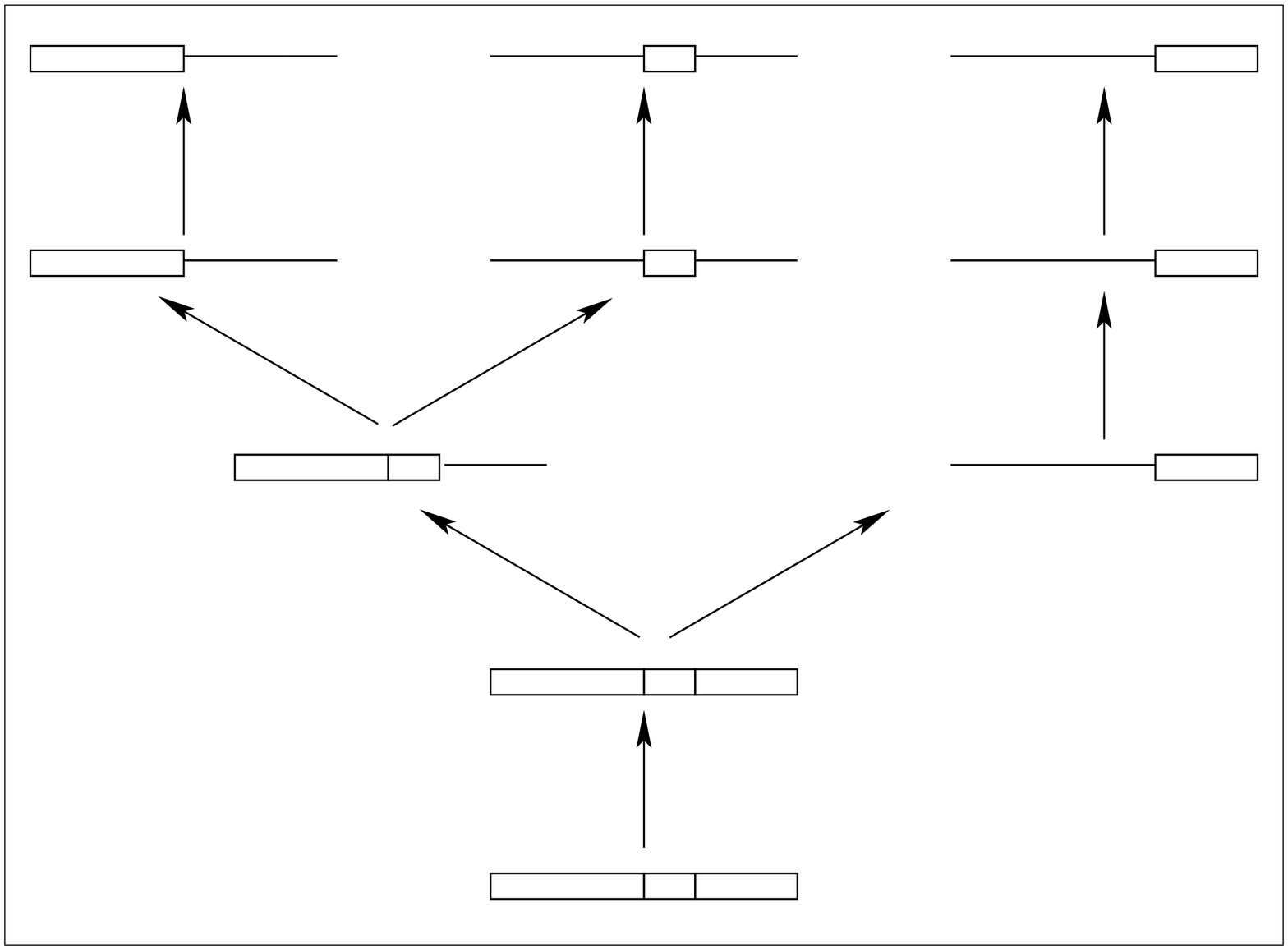} \\[6mm]
  \includegraphics[scale=0.3]{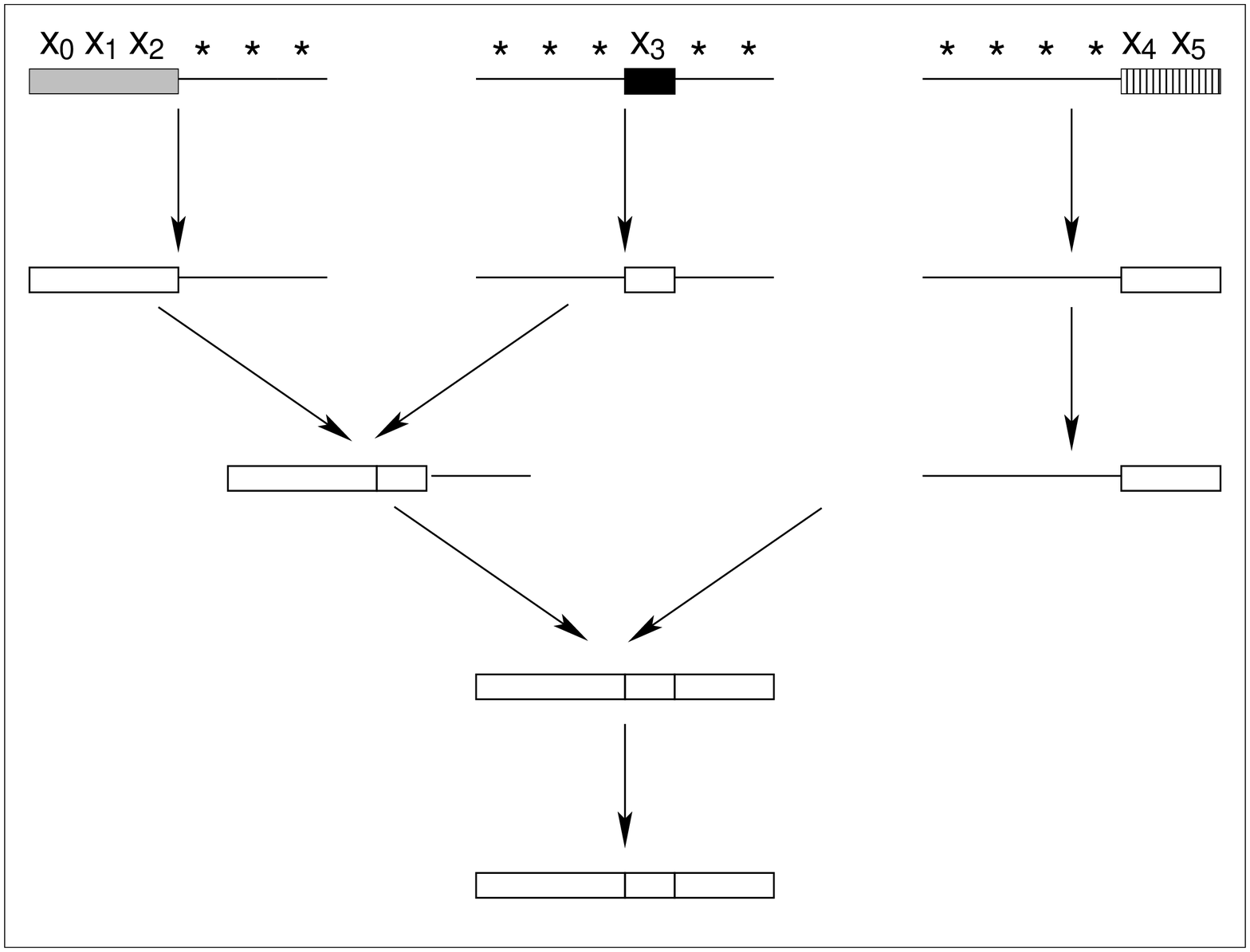} \\[6mm]
  \includegraphics[scale=0.3]{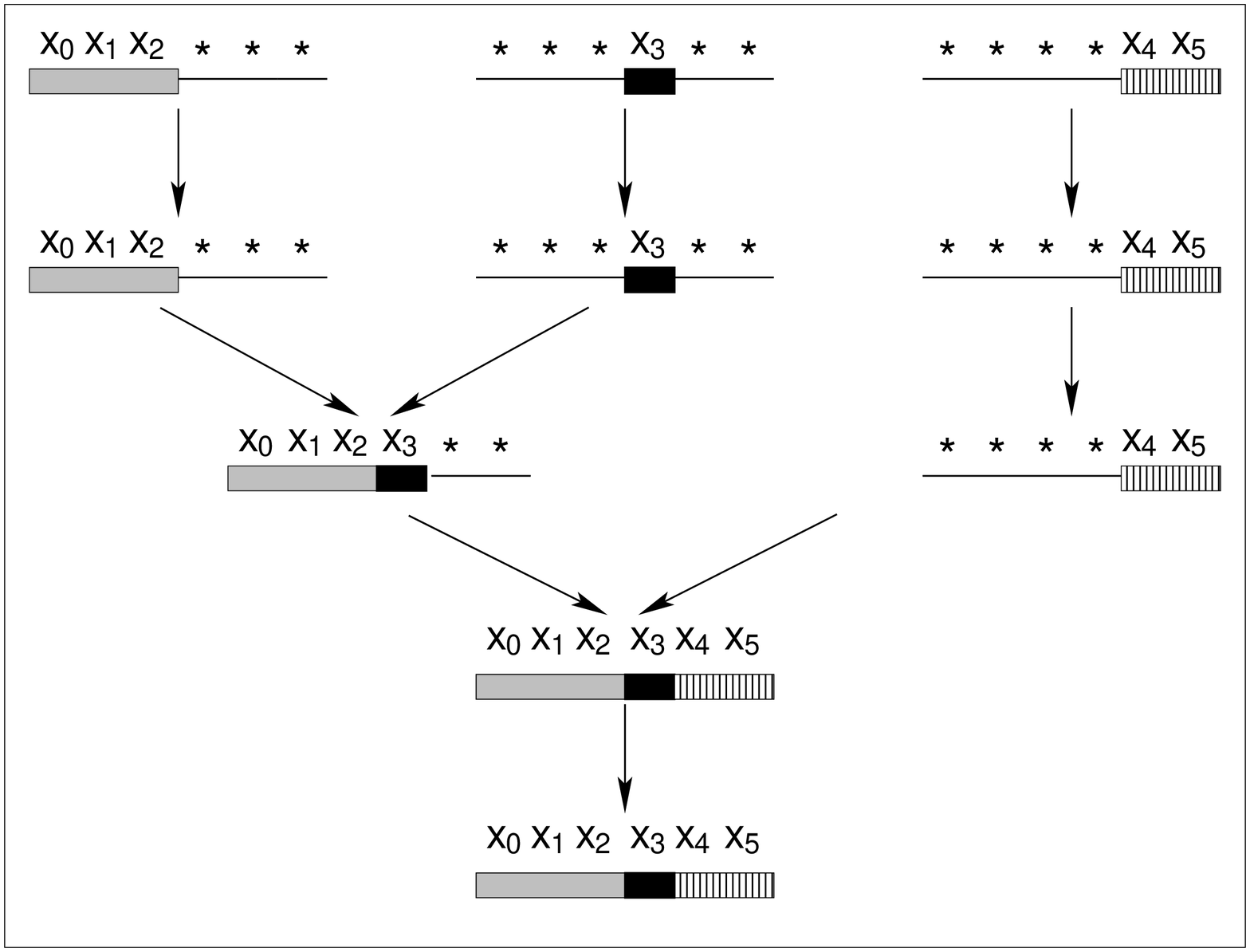}
\end{center} 
\caption{\label{fig:bigtree}
Construction of a random individual at time $t$, together with
its ancestry. Top: partitioning of sites (backward; step (A1)); middle:
assignment of colours and letters at the top; bottom: propagation of
colours and letters downwards. The middle and bottom panels
together correspond to the simultaneous performance of steps (A2) and (A3).
In this example, there are no coalescence events, so all
partitions are  ordered. As a consequence, this realisation of the
partitioning process 
$\{\Sigma_{\tau}\}_{0 \leqslant \tau \leqslant t}$ is a tree and, 
at the same time,
a realisation of
the segmentation process $\{F_{\tau}\}_{0 \leqslant \tau \leqslant t}$ of Section
\ref{subsec:segmentation}.} 
\end{figure}
\begin{enumerate}
\item[(A1)]
Run  $\{\Sigma_\tau\}_{\tau\in \NN_0}$  until $\tau=t$. 
$\Sigma_t$ tells us how the ancestral material of  our individual
is partitioned into parents at forward time $0$ (in Figure~\ref{fig:bigtree}, 
this
is the top of the tree).
\item[(A2)]
Assign a different colour to each part of $\Sigma_t$.
The colours are for illustration; each colour corresponds to 
one individual from the initial
population (at $t=0$), chosen uniformly \emph{without} replacement.
Equivalently, one may sample a parent from the initial population
\emph{with} replacement for every part of $\Sigma_{t-1}'$.
(The latter is  more convenient and will be favoured in what
follows.)
In any case, 
every site receives a colour, which is propagated
downwards. This results in the 
present individual pieced together from segments of different colours that 
correspond to different parental individuals.

\item[(A3)]
Assign a letter to every site at $t=0$ (i.e., $\tau=t$). 
By (A2), this entails that the type 
for part $\sigma_j$ of $\Sigma_{t-1}'$ is drawn from $\pi^{}_{\sigma_j}.
\widehat{Z}_0$, independently for every element of the partition. 
Like the colours, 
the letters are attached 
to the sites once and for all, and thus propagated downwards, i.e., down to
$\Sigma_0$.
\end{enumerate}

As a consequence of (A2) and (A3), conditional on $\Sigma'_{t-1} = \sigma 
= \{\sigma_1 , \dots , \sigma_k\}$, the type distribution at 
present (that is, at forward time $t$) is 
$:\!\!(\pi^{}_{\sigma_1} . \widehat Z_0) \otimes \dots \otimes 
(\pi^{}_{\sigma_k} . \widehat Z_0)\!\!:$, where $:\! \ldots \! :$ means 
that the factors are ordered as in $X$.
Denoting by $\Xi_t$ the type at forward time $t$, we thus have
\begin{equation}\label{eq:typedistr}
\PP(\Sigma_{t-1}'=\{\sigma_1 , \dots , \sigma_k\},\Xi_t=x) =
\PP(\Sigma_{t-1}'=\{\sigma_1 , \dots , \sigma_k\}) 
:\!(\pi^{}_{\sigma_1} . \widehat Z_0) \otimes \dots 
\otimes (\pi^{}_{\sigma_k} . 
   \widehat Z_0)\! : (x)\,.
\end{equation}
Eq.~\eqref{eq:typedistr} gives the marginal 
distribution (of  partition and type) for every \emph{single}
individual in a sample, or in the entire population.
Due to coalescence events, however, the individuals
in a finite population are correlated, and the joint distribution is
a difficult matter. (This is investigated within the framework of the
\emph{ancestral recombination graph}, which traces back the genealogy
of a \emph{sample} of individuals; compare \citealt[Chap.~7.2]{Wakeley},
\citealt[Chap.~3.4]{Durrett}, and
Section~\ref{sec:discussion}).

Our  goal here is a somewhat simpler one, namely, 
the distribution of types and ancestries in the 
$N\to\infty$
limit (under \emph{strong} recombination). In this limit,
the partitioning process simplifies substantially due to the 
following result.

\begin{lemma}\label{lem:refine}
Let $\Omega_t$ be the event that 
$\Sigma_{\tau} = \Sigma_{\tau-1}'$ for $1\leqslant \tau \leqslant t$;
that is, no coalescence occurs until (backward) time $t$, or,
equivalently,  $\{\Sigma_\tau\}_{0 \leqslant\tau \leqslant t}$ 
is a process of progressive refinements of ordered partitions.
For 
every fixed finite $t$, one  
has $\mathbb{P}(\Omega_t) \geqslant 1 -n(n+1)t/2N+ \cO(1/N^2)$.
\end{lemma}

\begin{proof}
For every $\tau$, in step (S),
$\Sigma'_{\tau}$ is obtained from $\Sigma_{\tau}$ as a 
refinement. It is thus clear that 
$\{\Sigma_{\tau}\}_{0 \leqslant\tau \leqslant t}$ is a process of 
progressive refinements
(and hence of ordered partitions) if and only if
$\Sigma_{\tau}=\Sigma'_{\tau-1}$ for $1 \leqslant \tau \leqslant t$.
If  $\Sigma'_{\tau-1}$ has $k$ parts, then the probability that each is
assigned to a different parent in the  coalescence step leading
to $\Sigma_{\tau}$  is 
\begin{equation}\label{eq:q_k}
q^{}_k := 1 \cdot \left( 1-\frac{1}{N} \right) \cdot 
\left( 1-\frac{2}{N}\right) \cdot \ldots \cdot \left( 1 - \frac{k-1}{N} \right).
\end{equation}
Obviously,
\begin{equation} \label{eq:q}
q^{}_k \geqslant q^{}_{n+1}= 
1 - \frac{n (n+1)}{2N} + \mathcal \cO(1/N^2)
\end{equation}
because $k \leqslant |S|=n +1$. As a 
consequence, 
\[
\PP(\Omega_t)\geqslant q_{n+1}^t= 1-\frac{ n (n+1) t}{2N} + 
\cO(1/N^2) 
\]
 for 
every fixed  finite $t$.  \qed
\end{proof}

Note that Lemma~\ref{lem:refine} implies that, for any finite $t$, 
coalescence events are absent in the $N \to \infty$ limit 
and the ancestry  is a tree -- in line with intuition, and as in 
Figure~\ref{fig:wfances} (left), and in Figure~\ref{fig:bigtree}.
Note also that
Lemma~\ref{lem:refine} holds for any finite $t$, but not for $t \to \infty$,
in the same spirit as the law of large numbers in Prop.~\ref{prop:IPL}.

\subsection{Segments and the segmentation process}
\label{subsec:segmentation}
Since, as we have just seen, we only have to deal with ordered
partitions (with probability one for any finite $t$ as $N \to \infty$), 
we can introduce a simplifying notation for the partitions
that is based on links rather than on sites.
 This is because ordered partitions are in one-to-one
correspondence with the subsets of $L$ as follows. 
As in \cite{MB}, let  $G=\{\alpha^{}_1,\ldots,\alpha^{}_{\vert G\vert}\}\subseteq L$,
with $\alpha^{}_1<\alpha^{}_2<\cdots <\alpha^{}_{\vert G\vert}$, an ordering
which we will assume implicitly from now on.
Let then ${\mathcal S}(\varnothing):=\{S\}$ and,
for $G \neq \varnothing$, let ${\mathcal S}(G) := 
\{\sigma_1, \sigma_2, \ldots, \sigma_{\vert G\vert+1}\}$ denote the
ordered partition of $S$ with parts 
\begin{equation}\label{eq:parts}
\sigma_1:= \{0,\ldots, \lfloor \alpha_1 \rfloor \}, \,
\sigma_2:= \{\lceil \alpha_1 \rceil, \ldots, \lfloor \alpha_2 \rfloor \},
\ldots, \sigma_{\vert G\vert+1} := 
\{\lceil \alpha_{\vert G\vert} \rceil, \ldots, n \}.
\end{equation}
In particular, ${\mathcal S}(L)= \bigl \{\{0\},\ldots, \{n\} \bigr \}$. It is clear that 
${\mathcal S}(H)$ is a
refinement of ${\mathcal S}(G)$ if and only if $G \subseteq H$.
It is also obvious that $\cS$ defines a bijection; its inverse,
$\psi := \cS^{-1}$, associates with every \emph{ordered} partition
of $S$ the corresponding subset of $L$, so that
$\psi(\cS(G))=G$ for all $G \subseteq L$.  

We now define the associated ordered partitions ${\mathcal L}^{}_G$
of $L \setminus G$. Let $\mathcal{L}_\varnothing=\{L\}$ and, for
$G \neq \varnothing$, set
(cf.~Figure~\ref{fig:subsystems}):
\begin{equation}\label{eq:L_G}
\begin{split}
  \widetilde {\mathcal L}^{}_G & :=  \Bigl \{ \big \{\alpha \in L: \tfrac{1}{2} \leqslant \alpha < \alpha^{}_1 \},   
          \big \{\alpha \in L: \alpha^{}_{1} < \alpha <\alpha^{}_{2} \big \},
   \ldots, 
    \big\{\alpha \in L: \alpha^{}_{\vert G\vert} < \alpha \leqslant \tfrac{2n-1}{2} \big\} \Bigr \}   \,, \\
  {\mathcal L}^{}_G & :=  \widetilde {\mathcal L}^{}_G \setminus \{\varnothing\}\,.
\end{split}
\end{equation}
That is, $\mathcal{L}^{}_G$ is the ordered partition of $L \setminus G$
that  holds the segments  (in the sense of  contiguous sets of links)
that arise when recombination has occurred 
at all links in $G$;
in particular, 
$\mathcal{L}_L= \{ \}$.

\begin{figure}[h]
  \begin{center} 
  \includegraphics[width=7cm,height=4.5cm]{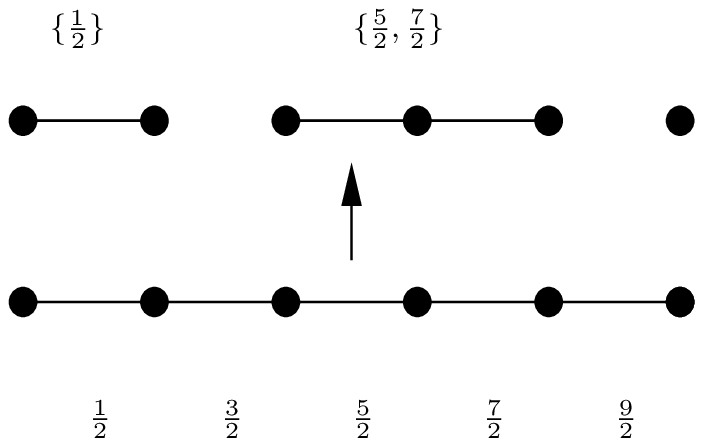}
  \end{center}  
\caption{The segments induced by $G=\left\{\frac{3}{2},\frac{9}{2}\right\}$
in the case $L=\left\{\frac{1}{2},\ldots,\frac{9}{2}\right\}$
(i.e., $\widetilde {\mathcal L}^{}_G = 
\bigl \{ \{\frac{1}{2}\},\{\frac{5}{2},\frac{7}{2}\},\varnothing\bigr \}$, $\mathcal{L}_G = \bigl \{ \{\frac{1}{2}\},\{\frac{5}{2},\frac{7}{2}\}\bigr \}$). 
$\mathcal{L}_G$ corresponds to the ordered partion
${\mathcal S}(G)=\{ \{0,1\},\{2,3,4\},\{5\}\}$ of $S=\{0,1,2,3,4,5\}$.}
\label{fig:subsystems} 
\end{figure}

Let us now consider the following process 
of progressive segmentation (which will turn out to coincide
with the 
$N\to \infty$ limit of the partitioning process for any finite time).

\begin{definition}[Segmentation process] \label{defin:segmentation}
The {\em segmentation process} is the discrete-time Markov chain 
$\{F_{\tau}\}_{\tau \in \NN_0}$, where $F_{\tau}$ takes values in the
power set of $L$ according to the following rules.
Start with $F_0= \varnothing$ and recall that $\cL_{\varnothing}=\{L\}$.
If $F_{\tau} =G$, choose either none or
one link in every segment, according to the following rule.
From segment $I$ of $\mathcal{L}_G$,
independently of all other segments, either no link is chosen
(probability $1-\sum_{\alpha\in I}\varrho^{}_{\alpha}$), or a
single link is chosen, namely  link $\alpha\in I$ with probability 
$\varrho^{}_{\alpha}$.
Then $F_{\tau+1}$ is  the union of $G$ with the set of all newly chosen links.
That is,
\begin{equation}\label{segmentprocess}
  F_{\tau+1} = F_{\tau} \cup A, \quad \text{where} \quad A =(\bigcup_{I\in\mathcal{L}_G} A^{}_I)
\end{equation}
and
\begin{equation}\label{segmentprocess2}
  A^{}_I =
\begin{cases}
  \varnothing, & \text{with probability } 1-\sum_{\alpha\in I}\varrho^{}_{\alpha} \, , \\
  \{\alpha\}, &  \text{with probability } \varrho^{}_{\alpha} \text{, for all } \alpha\in I \, .
\end{cases}
\end{equation}
\end{definition}
Clearly, picking a link corresponds to recombination, and
$F_{\tau}$ is the set of links that have been cut until time $\tau$.
Note that, as in the Wright-Fisher model with recombination (and its 
deterministic limit),
the links are not, in general, independent: At most one link in a
given segment may be cut in one time step; cutting of one link
prevents cutting of any other link in the same generation.
However, the backward point
of view adopted here reveals \emph{(conditional) independence of the 
individual segments}
once they arise. Put differently,  links are independent as soon as
they are on different segments. This is analogous to the conditional
independence of offspring inviduals in branching processes 
and will turn out as the golden key to the solution.

The connection of the segmentation process
with the partitioning process  can now be clarified (we use upper indices
once more to denote the dependence on population size):

\begin{proposition}
\label{coro:Sigma_F}
Let $t\geqslant 0$ be arbitrary but fixed. The law of 
$\{\Sigma^{(N)}_{\tau} \}_{ 0 \leqslant \tau \leqslant t}$ then agrees with that of  
$\{F_{\tau}\}_{ 0 \leqslant \tau \leqslant t}$ up to $\cO(1/N)$, 
provided  $\{\Sigma^{(N)}_{\tau}=\sigma\}$
is put in bijective correspondence with $\{F_{\tau}=\psi(\sigma)\}$ 
for all ordered
partitions $\sigma$ and
$0 \leqslant \tau \leqslant t$. For individual time points $\tau \leqslant t$,
this
implies specifically that
\begin{equation} \label{eq:Sigma_F_t}
  \PP(\Sigma^{(N)}_{\tau}=\sigma) =
  \begin{cases} \PP(F_{\tau}=\psi(\sigma)) + \cO(1/N), & \text{if } \sigma  
      \text{ is an ordered partition}, \\
                \cO(1/N),         & \text{otherwise}  \end{cases}
\end{equation}
for every $\sigma \in \Pi(S)$, with $\psi=\cS^{-1}$ as defined after \eqref{eq:parts}.
\end{proposition}

\begin{proof}
It is clear that, under the above identification,
the initial conditions ($\Sigma^{(N)}_{0}=\{S\}$
and $F_0=\psi(\{S\})=\varnothing$) agree. It is also clear that,
if $\Sigma^{(N)}_{\tau}=\cS(G)$ for some $G \subseteq L$,
then $(\Sigma^{(N)})'_{\tau}$ follows the same
law as $F_{\tau+1}$ given $F_{\tau}=G$ (by step (S) and 
Def.~\ref{defin:segmentation}).

Consider now $\{\Sigma^{(N)}_{\tau} \}_{ 0 \leqslant \tau \leqslant t}$
conditional on $\Omega_t$. By the above observation together with
Lemma \ref{lem:refine},  the law of
the conditional process  may be understood as follows.
Run $\{F_{\tau}\}_{ 0 \leqslant \tau \leqslant t}$, but in every step $\tau$,
kill the process with probability $1-q_{k}$ (from \eqref{eq:q_k})
if $|F_{\tau}|=k$.
The law of the surviving process then is the law of 
$\{\Sigma^{(N)}_{\tau} \, \mid \, \Omega_t\}_{ 0 \leqslant \tau \leqslant t}$.
Since the killing probability up to time $t$
is $1-\PP(\Omega_t)=\cO(1/N)$ 
(see Lemma \ref{lem:refine}),  the law of 
$\{\Sigma^{(N)}_{\tau}\, \mid \, \Omega_t \}_{ 0 \leqslant \tau \leqslant t}$
agrees with that of $\{F_{\tau}\}_{ 0 \leqslant \tau \leqslant t}$ up to
$\cO(1/N)$. Finally, using $1-\PP(\Omega_t)=\cO(1/N)$ once more
yields the claim.  
\end{proof}   \qed

Let us remark that \eqref{eq:Sigma_F_t}  (evaluated at time 
$\tau-1$) also entails that
\begin{equation}
\label{eq:Sigma_F_t-1}
  \PP((\Sigma^{(N)})'_{\tau-1})=\sigma) =
\begin{cases} \PP(F_{\tau}=\psi(\sigma)) + \cO(1/N), & \text{if } \sigma  
      \text{ is an ordered partition,} \\
                \cO(1/N),         & \text{otherwise  }  \end{cases}
\end{equation}
since $(\Sigma^{(N)})'_{\tau-1}$ is obtained from $\Sigma^{(N)}_{\tau-1}$
according to the same law as $F_{\tau}$ from $F_{\tau-1}$, provided
$\Sigma^{(N)}_{\tau-1}$ is an ordered partition.
Let us also remark that, for \emph{finite} $N$,
$\{\Sigma^{(N)}_{\tau} \mid \Omega_t\}_{0 \leqslant \tau \leqslant t}$ 
only agrees with $\{F_{\tau}\}_{0 \leqslant \tau \leqslant t}$
up to a probability of $\cO(1/N)$. This is due to a bias, in a
finite population,  towards partitions with fewer parts, since
these bear less
risk of coalescence.

Before we proceed, let us note another elementary, but crucial property of
the segmentation process. Consider the segmentation process 
$\{F^{(\tilde L)}_{\tau}\}_{\tau \in \NN_0}$ on a \emph{contiguous} subset
$\tilde L$ of  $L$. Here $\{F^{(\tilde L)}_{\tau}\}_{\tau \in \NN_0}$
is  defined in the same way as  $\{F_{\tau}\}_{\tau \in \NN_0}$ but with $L$ replaced
by $\tilde L$, and based on the
recombination probabilities
$\varrho_{\alpha}$, $\alpha \in \tilde L$, alone.
Here, the upper index now indicates dependence on the
(sub-) set of links, which we may omit if $\tilde L = L$; that is, 
$F_{\tau}^{(L)}=F_{\tau}$. Likewise, we will denote by 
$\mathcal L^{(\tilde L)}_G$, $G \subseteq \tilde L$, the partition of 
$\tilde L \setminus G$ defined in analogy with \eqref{eq:L_G}, with
$L$ replaced by $\tilde L$. 
We then have the following important fact.

\begin{proposition}[marginalisation property]\label{prop:marg}
Let $\tilde L$ be a contiguous subset of $L$. 
The process $\{F_{\tau}^{(\tilde L)}\}_{\tau \in \NN_0}$ then is the 
marginal version of 
$\{F_{\tau}^{(L)}\}_{\tau \in \NN_0}$ with respect to the links
in $\tilde L$, that is,
for all $G \subseteq \tilde L$ and all $\tau \in \NN_0$, we have
\[
 \PP(F_{\tau}^{(\tilde L)}=G) = \sum_{H \subseteq L \setminus \tilde L} 
  \PP(F_{\tau}^{(L)}=G \cup H).
\]
\end{proposition}

\begin{proof}
We will prove the claim by  showing that, for every $G \subseteq \tilde L$ and 
$H \subseteq L \setminus \tilde L$,  the set $A^{(\tilde L)}$ of links 
picked from
the segments of $\mathcal L^{(\tilde L)}_G$ if $F_{\tau}^{(\tilde L)}=G$
follows the same law as $A^{(L)}\cap \tilde L$, where $A^{(L)}$ is the set of links picked from
the segments of $\mathcal L^{(L)}_{G\cup H}$  if $F_{\tau}^{(L)}=G\cup H$
(according to \eqref{segmentprocess} and \eqref{segmentprocess2},
and likewise for $\tilde L$).
Due to the independence of the segments (within $\mathcal L^{(\tilde L)}_G$ 
and within $\mathcal L^{(L)}_{G\cup H}$), we may consider these segments
separately.
Segments that are contained in both $\mathcal L^{(\tilde L)}_G$ and 
$\mathcal L^{(L)}_{G\cup H}$ contribute identically to
$A^{(\tilde L)}$ and $A^{(L)}\cap \tilde L$
by construction.
Segments $I \in \mathcal L^{(L)}_{G\cup H}$ with $I \cap \tilde L = \varnothing$
do not contribute to $A^{(L)}\cap \tilde L$ and are independent of those
in $\mathcal L^{(\tilde L)}_G$ and thus of $A^{(\tilde L)}$.
We are left to consider segments $\tilde I \in \mathcal L^{(\tilde L)}_G$
with $\tilde I \subsetneq  I \in \mathcal L^{(L)}_{G\cup H}$ for some $I$.
But here the probability to pick any $\alpha \in \tilde I$ 
(for $A^{(\tilde L)}$ ) is the same as
picking this $\alpha$ from $I$ for $A^{(L)}\cap \tilde L$,
namely, $\varrho_{\alpha}$, which completes the proof. \qed
\end{proof}


We are now ready to state the main result of this section.
\begin{theorem}[type distribution via ancestral process]\label{thm:typedistr}
Consider a sequence of Wright-Fisher models with single-crossover
recombination and increasing population size $N$. 
Let the initial states be such that
$\lim_{N\rightarrow\infty} \widehat{Z}_0^{(N)} = p^{}_0$.
The type distribution of any given individual for any finite 
$t \in \NN_0$ converges to
$\sum_{G \subseteq L} \PP(F_t=G) R_G(p^{}_0)$ 
as $N \to \infty$. For the
composition of the population, we have 
\begin{equation}\label{eq:LLN-back}
\lim_{N \to \infty}  \widehat Z_t^{(N)} = \sum_{G \subseteq L} 
\PP(F_t=G) R_G(p_0) \quad \text{in mean square}. 
\end{equation}
\end{theorem}
Clearly, \eqref{eq:LLN-back} is again a law of large numbers,
analogous to the infinite-population limit of Prop.~\ref{prop:IPL},
but this time expressed in terms of the \emph{backward} process; the
connection will be exploited in the next section. 
As in Remark \ref{rem:interval}, the result again carries over to
finite time intervals. 

\begin{proof}[of Theorem \ref{thm:typedistr}]
We prove the theorem by considering the joint distribution of
partitions and types, $\Sigma'_{t-1}$ and $\Xi_t$,   as in \eqref{eq:typedistr}.
We will omit the dependence on
$N$ for ease of notation.
For a given single individual, 
we first  rewrite the joint probabilities as
\begin{equation} \label{eq:cond1}
\PP(\Sigma'_{t-1}=\sigma, \Xi_t=x) =
  \PP(\Sigma_{t-1}'=\sigma) \, \PP(\Xi_t=x  \mid  \Sigma_{t-1}'=\sigma).
\end{equation}
As to the first term on the right-hand side,
recall that, by \eqref{eq:Sigma_F_t-1}, the only partitions 
that survive in the limit are $\cS(G)$, $G \subseteq L$,
for which  $\PP(\Sigma_{t-1}'=\cS(G)) = 
\PP(F_t=G) + \cO(1/N)$. As to the second term, 
\eqref{eq:typedistr}   tells
us that the  type distribution 
corresponding to $\cS(G)$ is
\[
(\pi_{\sigma_1}.\widehat Z_0) \otimes \dots \otimes (\pi_{\sigma_{|G|+1}}.\widehat Z_0)
=R_G(\widehat Z_0)
\] 
with $\sigma_1, \ldots, \sigma_{|G|+1}$
of \eqref{eq:parts}. But $R_G(\widehat Z_0)$
converges to $R_G(p_0)$
by assumption. By Lemma \ref{lem:refine}, 
in the limit $N \to \infty$, \eqref{eq:cond1} thus
becomes
\begin{equation}\label{single_limit}
\lim_{N \to \infty}
\PP(\Sigma'_{t-1}=\sigma, \Xi_t=x) =
\begin{cases} \PP(F_t=G) (R_G(p_0))(x) & \text{if } 
                          \sigma=\cS(G) \\
              0 & \text{otherwise},
\end{cases}
\end{equation}
from which the type distribution for single individuals
follows by marginalisation over $\Sigma'_{t-1}$.

As to the population, let individuals be numbered $1,2,\ldots,N$,
and let $\Sigma'_{t-1,i}$ and $\Xi_{t,i}$ be the  partition
at (backward) time $t\!-\!1$ (after splitting), and the type at (forward) time $t$,
respectively, of individual $i$, $1 \leqslant i \leqslant N$ 
(so that the above
$\Sigma'_{t-1}$ and  $\Xi_{t}$ may be identified with  
$\Sigma'_{t-1,1}$ and  $\Xi_{t,1}$). The $\Sigma'_{t-1,i}$ 
are identically distributed (across $i$), but not independent
(they are correlated due to common ancestry); the same holds
for the $\Xi_{t,i}$. Clearly,
\[
\widehat Z_t(x) = \frac{1}{N} \sum_{i=1}^N \one \{\Xi_{t,i}=x\}, \quad x \in X,
\]
where $\one \{\ldots\}$ denotes the indicator function for the event
in question.
We will show that, for all $x \in X$ and $\sigma \in \Pi(S)$,
\begin{equation} \label{eq:willshow}
\lim_{N \to \infty} \Big (\frac{1}{N} 
\sum_{i=1}^N \one \{\Sigma'_{t-1,i}=\sigma,\Xi_{t,i}=x\}
- \PP(\Sigma'_{t-1,1}=\sigma,\Xi_{t,1}=x) \Big ) =0
\end{equation}
in mean square. Eq.~\eqref{eq:LLN-back} then follows via summation over
all $\sigma \in \Pi(S)$, together with the result for single
indviduals, which tells us that $\PP(\{\Xi_{t,1}=x\})
\xrightarrow{N \to \infty} \PP(F_t=G) \big (R_G(p_0)\big )(x)$.

To establish \eqref{eq:willshow}, it is
sufficient to show that the covariance of 
$\one \{\Sigma'_{t-1,1}=\sigma,\Xi_{t,1}=x\}$ and
$\one \{\Sigma'_{t-1,2}=\sigma,\Xi_{t,2}=x\}$ is $\cO(1/N)$;
due to exchangeability, this then carries over to arbitrary pairs
of individuals. 
Let $\widetilde \Omega_t$
be the event that no coalescence happens between ancestors of
individual 1 and those of individual 2 until time $t$
(while coalescences between ancestors
of the same individual are allowed).  
In analogy with \eqref{eq:q}, the 
probability of no such coalescence in a single time step  is bounded
from below by
\begin{equation}\label{eq:tildeq}
\tilde q := \Big ( 1-\frac{n+1}{N} \Big)^{n+1} = 1-\frac{(n+1)^2}{N} 
+ \cO(1/N^2)
\end{equation}
since each individual has at most $|S|=n+1$ ancestors.
Thus 
\[
\mathbb{P}(\widetilde \Omega_t) \geqslant \tilde q^t=
1 -\frac{(n+1)^2t}{N} + \cO(1/N^2)
\]
for every finite $t$ as $N\to \infty$.
We now consider $\{\Sigma'_{\tau,1},\Sigma'_{\tau,2} \, \mid \, 
\widetilde \Omega_{t-1}\}_{0 \leqslant \tau \leqslant t-1}$. 
Arguing in a similar way as
in the proof of Prop.~\ref{coro:Sigma_F},
we find that the law of the conditional joint process
agrees with that of two \emph{independent} copies of 
$\{\Sigma'_{\tau}\}_{0 \leqslant \tau \leqslant t-1}$
as long as there is no
common ancestry between parts of $\Sigma'_{\tau,1}$
and those of $\Sigma'_{\tau,2}$. This is the case with probability
$\PP(\tilde \Omega_{t-1})$, which deviates from $1$ by $\cO(1/N)$, so that
\begin{equation}\label{eq:sigma}
\PP (\Sigma'_{t-1,1}=\sigma,\Sigma'_{t-1,2}=\sigma 
 \mid  \widetilde \Omega_{t-1})  =
\big (\PP(\Sigma'_{t-1,1}=\sigma ) \big )^2 + \cO(1/N). 
\end{equation}
Next, on $\widetilde \Omega_{t-1}$, the parts of $\Sigma'_{t-1,1}$ 
and of $\Sigma'_{t-1,2}$
pick their types \emph{independently}, so that 
\begin{equation}\label{eq:x}
\PP(\Xi_{t,1}=x,\Xi_{t,2}=x  \mid 
\Sigma'_{t-1,1}=\sigma,\Sigma'_{t-1,2}=\sigma, \widetilde \Omega_{t-1}) =
\big ( \PP(\Xi_{t,1}=x  \mid 
\Sigma'_{t-1,1}=\sigma)\big )^2. 
\end{equation} 
Taking together $1-\PP(\widetilde \Omega_{t-1})=\cO(1/N)$,
\eqref{eq:sigma} and \eqref{eq:x}, we get
\[
\begin{split}
 \PP (\Sigma'_{t-1,1}=\sigma, & \Sigma'_{t-1,2}=\sigma, \Xi_{t,1}=x, \Xi_{t,2}=x )  \\
& =  \PP (\Sigma'_{t-1,1}=\sigma,\Sigma'_{t-1,2}=\sigma,\Xi_{t,1}=x, \Xi_{t,2}=x \mid
\widetilde \Omega_{t-1}) + \cO(1/N)  \\
& =  \PP (\Xi_{t,1}=x, \Xi_{t,2}=x \mid
\Sigma'_{t-1,1}=\sigma,\Sigma'_{t-1,2}=\sigma,\widetilde \Omega_{t-1}) \\
& \qquad \qquad \times \PP (\Sigma'_{t-1,1}=\sigma,\Sigma'_{t-1,2}=\sigma \mid  \widetilde \Omega_{t-1})
+ \cO(1/N)  \\
 & =  \big ( \PP (\Xi_{t,1}=x, \Sigma'_{t-1,1}=\sigma ) \big )^2 + \cO(1/N),
\end{split} 
\]
so that indeed
\[
\text{Cov}(\one \{\Sigma'_{t-1,1}=\sigma, \Xi_{t,1}=x\},
\one \{\Sigma'_{t-1,2}=\sigma, \Xi_{t,2}=x\}) = \cO(1/N),
\]
which establishes \eqref{eq:willshow} and proves the claim. \qed
\end{proof}

\subsection{Connection with the deterministic dynamical system}
A main result now is
\begin{theorem}\label{thm:segisa}
 For all $G\subseteq L$ and all $\tau\geqslant 0$, we have 
\[
 \mathbb{P}(F_{\tau} =G) = a_G(\tau) \, .
\]
\end{theorem}
We give two proofs that result in different and mutually complementary insight.
The first uses a general argument, the second
a concrete calculation.
\begin{proof}[First proof of Theorem $\ref{thm:segisa}$]
  Compare the two laws of large numbers, Prop.~\ref{prop:IPL} and
  Theorem~\ref{thm:typedistr}. The claim is obvious via comparison of
  coefficients. The latter is justified by the
  following observation (cf.\ the argument in the proof of Theorem~3 in
  \citealt{discretereco}). For generic $p_0$ and generic $X_i$, the
  vectors $R_G(p_0)$ with $G \subseteq L$ are the extremal vectors of
  the closed simplex $\text{conv}\{R_K (p_0) \mid K \subseteq L\}$,
where $\text{conv}$ denotes the convex hull. They
  are the vectors that (generically) cannot be expressed as
  non-trivial convex combination within the simplex, and hence the
  vertices of the simplex (in cases with degeneracies, one reduces the
  simplex in the obvious way). \qed
\end{proof}

\begin{proof}[Second proof of Theorem $\ref{thm:segisa}$]
We simply show that $a_G(\tau)$ and $\PP(F_{\tau}=G)$ follow the same recursions
(with the same initial values).
To this end, recall that we have implied
\[
\mathbb{P}(F_{\tau} =G) = \mathbb{P}(F_{\tau} =G \mid F_0 = \varnothing).
\]
We then decompose the $\tau\!+\!1$ time steps into the initial step, followed
by an interval of $\tau$ steps (in the spirit of the Kolmogorov
backward equation)
to obtain
\begin{equation}\label{eq:segprocess}
\begin{split}
 \mathbb{P}(F^{(L)}_{\tau+1}&=G) = \mathbb{P}(F^{(L)}_{\tau+1} = G \mid F^{(L)}_0=\varnothing)
  = \sum_{H\subseteq G} \mathbb{P}(F^{(L)}_1 = H \mid F^{(L)}_0 =\varnothing)
\mathbb{P}(F^{(L)}_{\tau}=G \mid F^{(L)}_0 =H) \\
&= \PP(F_1^{(L)}=\varnothing\mid F_0^{(L)}=\varnothing) \mathbb{P}(F^{(L)}_{\tau}=G \mid F^{(L)}_0 =\varnothing) \\
& \qquad  + \sum_{\alpha\in G} \mathbb{P}(F^{(L)}_1 = \{\alpha\} \mid F^{(L)}_0 =\varnothing)
\mathbb{P}(F^{(L)}_{\tau} = G \mid F^{(L)}_0 = \{\alpha\}) \\
&=  (1-\sum_{\alpha \in L} \varrho_{\alpha}) \mathbb{P}(F^{(L)}_{\tau}=G)
+ \sum_{\alpha\in G} \varrho^{}_{\alpha} \mathbb{P}(F_{\tau}^{(L_{<\alpha})} = G_{<\alpha})
\mathbb{P}(F_{\tau}^{(L_{>\alpha})} = G_{>\alpha}) \, .
\end{split}
\end{equation}
In the last step, we have used
$\mathbb{P}(F^{(L)}_{\tau} = G\vert F^{(L)}_0 = \{\alpha\}) =
 \mathbb{P}(F_{\tau}^{(L_{<\alpha})} = G_{<\alpha})
\mathbb{P}(F_{\tau}^{(L_{>\alpha})} = G_{>\alpha})$, which is due to the conditional
independence of segments according to Def.~\ref{defin:segmentation}.
Knowing Prop.~\ref{prop:marg} (for $L'=L_{< \alpha}$
and hence $L \setminus L' = L_{>\alpha}$), 
the recursion~\eqref{eq:segprocess} is identical to the one
in \eqref{nonlinrecur}; together with the identity of the initial
conditions,
\[
a^{(L)}_G(0) = \mathbb{P}(F^{(L)}_0 =G) = \delta^{}_{G,\varnothing},
\]
this proves the claim.  \qed
\end{proof}

It is important to note that the second proof does not rely on Theorem 
\ref{thm:typedistr}.
Theorem \ref{thm:segisa} could thus be used to establish Theorem 
\ref{thm:typedistr}
 by simply invoking
Prop.~\ref{prop:IPL} and the deterministic time evolution \eqref{ansatzsol}. However,
the independent proof of Theorem 2 bears the great
advantage that it only requires the stochastic arguments derived
in the current paper, thus making the argument self-contained
and independent of the knowledge of the deterministic dynamics
developed  in previous work via quite a different toolbox.

\subsection{Towards an explicit solution -- preparation and example}
Before we proceed, let us 
consider an important aspect of the segmentation process, 
namely, the probability
that nothing happens in one time step
given the current state is $G$. For any contiguous $\tilde L \subseteq L$,
let
\begin{equation}\label{lambdacoeff}
  \lambda^{(\tilde L)}_G := \mathbb{P}(F^{(\tilde L)}_{\tau+1} = G 
  \mid F^{(\tilde L)}_{\tau} = G) = 
  \prod_{I\in \mathcal{L}^{(\tilde L)}_G} (1-\sum_{\alpha\in I} \varrho^{}_{\alpha}), 
  \quad\text{ for } G\subseteq \tilde L \,.
\end{equation}
As before, we will omit  the upper index in the case of $\tilde L = L$,
that is, $\lambda^{(L)}_G = \lambda_G$. 
Since for every $I\in \mathcal{L}^{(\tilde L)}_G$ one has 
$\mathcal{L}^{(I)}_{\varnothing} = \{I\}$,
and thus $\lambda^{(I)}_{\varnothing}= 1-\sum_{\alpha\in I} \varrho^{}_{\alpha}$,
we can rewrite \eqref{lambdacoeff} as
\begin{equation}\label{lamkoeff}
  \lambda^{(\tilde L)}_G=
  \prod_{I\in \mathcal{L}^{(\tilde L)}_G}\lambda^{(I)}_{\varnothing} \, .
\end{equation}
The coefficients $\lambda_G$ have already been identified by
\cite{Bennett},  \cite{Lyubich}, and 
\cite{Kevin1,Kevin2}, as well as by \cite{discretereco} as the
generalised eigenvalues of the linearised deterministic dynamics.

Our aim now is to find a closed-form expression for  $\PP(F_{\tau}=G)$
for all $\tau$. Clearly, $\PP(F_{\tau}=G)$
is the sum over the probabilities for all paths that lead to the state
$F_{\tau}=G$. Each path of the segmentation process may be represented by
a tree, which we will call \emph{ancestral recombination tree} or
ART for short; 
we thus have to sum over the corresponding trees. Considering the trees
carefully will be the key to the solution. Let us illustrate this
by means of an example.
\begin{example}\label{ex:a13}
For four sites $S=\{0,1,2,3\}$ with the corresponding links $L=\{\frac{1}{2},\frac{3}{2},\frac{5}{2}\}$,
we consider $\mathbb{P}(F^{}_{\tau} = \{\frac{1}{2},\frac{3}{2}\})$,
as illustrated in Figure~\ref{fig:a13t}.
That is, we are concerned with all paths of the segmentation process
that lead to $F^{}_{\tau} = \{\frac{1}{2},\frac{3}{2}\}$.
The left tree captures the
path where link $\frac{1}{2}$
is the first to be cut, the second that link $\frac{3}{2}$ is cut first.
The $\lambda_G^j$ are the probabilities that nothing happens
to any of the current segments for $j$ time steps.
In the case where $\frac{3}{2}$ is the first 
event, the additional factor $\lambda^{(\{\frac{5}{2}\})}_{\varnothing}=(1-\varrho^{}_{\frac{5}{2}})$ 
is required to
guarantee that at the time of the second segmentation event (at link $\frac{1}{2}$),
the  segment that belongs to link $\frac{5}{2}$ remains unchanged
(the corresponding term in the other case is 
$\lambda^{(\varnothing)}_{\varnothing}=1$).
Finally, summing over all possible time combinations, one obtains
\begin{equation}\label{ex13eq}
\begin{split}
 \mathbb{P}(F^{}_{\tau}  = \{\tfrac{1}{2},\tfrac{3}{2}\}) & =  \varrho^{}_{\frac{1}{2}} \varrho^{}_{\frac{3}{2}}
        \sum_{k=0}^{\tau-2}\lambda^{k}_{\varnothing}\sum_{i=0}^{\tau-2-k}\lambda^{i}_{\frac{1}{2}}
                    \lambda^{{\tau-2-k-i}}_{\{\frac{1}{2},\frac{3}{2}\}} \\
  & + \varrho^{}_{\frac{1}{2}} \varrho^{}_{\frac{3}{2}} (1-\varrho^{}_{\frac{5}{2}})
        \sum_{k=0}^{\tau-2}\lambda^{k}_{\varnothing}\!\sum_{i=0}^{\tau-2-k}\lambda^{i}_{\frac{3}{2}}
                    \lambda^{\tau-2-k-i}_{\{\frac{1}{2},\frac{3}{2}\}} \, ,
\end{split}
\end{equation}
where the first double sum belongs to the left, the second to the right tree
of Figure~\ref{fig:a13t}.
Unsurprisingly, the same result is obtained by explicitly solving~\eqref{nonlinrecur}
(by the method established by \citealt{discretereco}), which demonstrates
once more that
$\mathbb{P}(F^{}_{\tau} = \{\tfrac{1}{2},\tfrac{3}{2}\}) = a^{}_{\{\frac{1}{2},\frac{3}{2}\}}(\tau)$, in line with Theorem~\ref{thm:segisa}.
\begin{figure}[h]
\begin{center}
  \includegraphics[width=5.5cm,height=4.5cm]{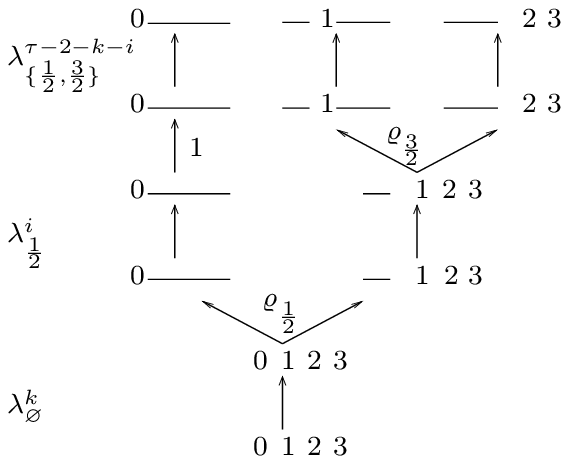}
  \qquad \quad
  \includegraphics[width=5.5cm,height=4.5cm]{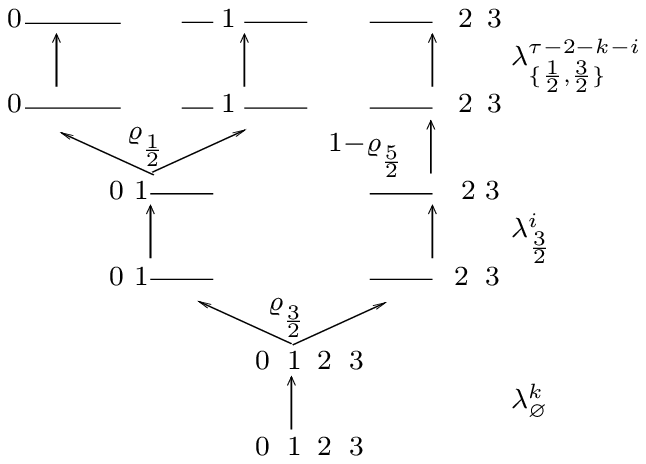}
\end{center}  
\caption{The two possible paths of the segmentation process of Example
\ref{ex:a13}
that lead to $F^{}_{\tau} = \{\tfrac{1}{2},\tfrac{3}{2}\}$.
The left panel refers to the first double sum of~\eqref{ex13eq}, the right one
to the second.}
\label{fig:a13t}
\end{figure}
\end{example}

$\mathbb{P}(F^{}_{\tau} =G)$
may be understood as a sum over both tree topologies and branch lengths,
i.e. we are concerned with all possible {\em ultrametric binary trees} 
that can be produced
by the segmentation process. (The trees will be explained in
more detail later. For the moment, recall that, in a binary tree, each
internal node has at most two offspring nodes. An ultrametric
tree is a tree whose branches are assigned lengths  such that all leaves have
the same distance from the root. For a review of metric trees,
see \citealt[Chap.~7]{SempleSteel}). 
In our case, the branch length corresponds to the number of time steps between
consecutive nodes, and  each internal node with two
offspring nodes corresponds to a recombination event.
We will now show  that (and how) it is sufficient to deal
with the corresponding {\em tree topologies} instead, which are
obtained by contracting consecutive edges connected by a node with
a single offspring into a \emph{single} edge and removing the branch
lengths. The result of this (many-to-one) operation 
is the topology of a \emph{full binary tree}, that is,
every internal node has exactly two offspring nodes (which may be
internal nodes or leaves). The probability for each topology then is
the sum of all probabilities of all the underlying original (ultrametric)
trees, that is, the probability for all possible combinations of
branch lengths (cf.\ the double sums in \eqref{ex13eq}.)
It will turn out that these sums 
may be evaluated explicitly for each topology, which is the reason
that this approach is useful. For Example~\ref{ex:a13},
this will (after a simple but lengthy calculation) result in
\begin{equation}\label{treeassum}
\begin{split}
\mathbb{P}(F^{}_{\tau} = \{\tfrac{1}{2},\tfrac{3}{2}\})  & =  \mathbb{P}(\text{tree $1$}) + \mathbb{P}(\text{tree $2$}) \\
   &= \Bigl((\lambda^{\tau}_{\{\frac{1}{2},\frac{3}{2}\}} - \lambda^{\tau}_{\varnothing})
   \tfrac{\varrho^{}_{\frac{1}{2}}}{\lambda^{}_{\{\frac{1}{2},\frac{3}{2}\}} - \lambda^{}_{\varnothing}}
   - (\lambda^{\tau}_{\frac{1}{2}} - \lambda^{\tau}_{\varnothing})\Bigr)\\
   & + \Bigl((\lambda^{\tau}_{\{\frac{1}{2},\frac{3}{2}\}} - \lambda^{\tau}_{\varnothing})
    \tfrac{\varrho^{}_{\frac{3}{2}}}{\lambda^{}_{\{\frac{1}{2},\frac{3}{2}\}} - \lambda^{}_{\varnothing}}
 - (\lambda^{\tau}_{\frac{3}{2}} - \lambda^{\tau}_{\varnothing})
    \tfrac{\varrho^{}_{\frac{3}{2}}}{\lambda^{}_{\frac{3}{2}} - \lambda^{}_{\varnothing}}\Bigr) \, ,
\end{split}
\end{equation}
where tree $1$ refers to the left  
and tree $2$ to the right panel in Figure~\ref{fig:a13t}.
\subsection{Ancestral tree topologies}
Our aim now is to assign probabilities to each of the possible
topologies   that have the elements of a 
given set $G$ as their internal nodes. Once the probabilities are known, 
$\mathbb{P}(F^{}_{\tau}=G)$ is obtained by summing over all compatible topologies.
Let us begin with a suitable definition for our tree topologies
(see Figure~\ref{fig:treetopology} for an illustration).

\begin{definition}
For $\varnothing\neq G\subseteq L$, a  \emph{tree topology}
is defined as $T:=(G,m)$, 
where $G$ signifies the set of internal nodes,  
$\gamma\in G$  designates the \emph{initial branching point}
of the tree, and,
in addition, $r$ is the \emph{root}.
The function $m$ is given by
\[
\begin{split}
  m: \, G &\longrightarrow G\cup \{r\} \\
  \alpha &\mapsto m(\alpha) \, ,
\end{split}
\]
and $m(\alpha)$ denotes the (unique) ancestor of the 
internal node $\alpha\in G$.
$m(\alpha)$ is an internal node except for $\alpha=\gamma$, where
$m(\gamma)=r$. 
We will assume throughout that $m$ is
\emph{tree-consistent}, that is,  the resulting structure
is a full binary tree topology. For $G=\varnothing$, 
the only tree topology is the empty tree (with no internal nodes).
\end{definition}
Thus, $T$ has the internal nodes $\alpha\in G$ and the set of internal edges
\[
\{(m(\alpha),\alpha) \mid \alpha\in G, m(\alpha) \neq r\},
\]
 as well
the external edge
 $(r,\gamma)$.
Note that  we have not included the (external) leaves 
(and corresponding external edges) in our definition
since they will never be required explicitly. Note also that, in the
context of phylogeny, the (canonical) root of the tree is what we call initial
branching point, cf.\ \citet{SempleSteel}. For an excellent account
of terminology and properties of trees, see \citet[Chap.~3]{GrossYellen}.

We will use the standard partial order on $T$, namely, 
$ \alpha \preccurlyeq \beta$ means that $\alpha$ is on the path from  
     $\gamma$ to $\beta$, i.e. $\alpha=m^i(\beta)$ 
        for some $i\in \{0,\ldots,\vert G\vert\}$.
Obviously,  $\gamma$ is the minimal element of $G$ with respect
to $\preccurlyeq$. Furthermore, $\alpha \prec \beta$ means that
$\alpha \preccurlyeq \beta$ with $\alpha \neq \beta$.

Note that the topology $T$ as such does not depend on $L$ except via
the requirement $G \subseteq L$; it may likewise represent
a realisation of 
a process $\{F^{(\tilde L)}_{s}\}_{0 \leqslant s
\leqslant \tau}$, restricted to
a (contiguous) subset $\tilde L$ of $L$,
provided $G \subseteq \tilde L$.  If we also specify the set of links,
say $\tilde L$, then each edge of a given tree topology can be
associated with a particular segment. Namely, for $T=(G,m)$ and
$\alpha\in G$, we associate with the  edge
$(m(\alpha),\alpha)$ the segment
\[
 I_{\alpha}^{(\tilde L)}(T) := K\in \mathcal{L}^{(\tilde L)}_{\{\beta: \beta  \prec \alpha\}} \, \text{ s.t. } \alpha\in K \, .
\]
In words, $I_{\alpha}^{(\tilde L)}(T)$ is the segment that
will receive its next cut at link $\alpha$ (given the topology $T$).
In particular, $I^{(\tilde L)}_{\gamma}(T)= \tilde L$ (independently of $T$).
An example  is given in Figure~\ref{fig:treetopology}.
From now on we will suppress the dependence  on $T$ and $\tilde L$
throughout 
and write $I^{}_{\alpha}$ instead of $I^{(\tilde L)}_{\alpha}(T)$.
\begin{figure}[h]
\begin{center}
\includegraphics[scale=1]{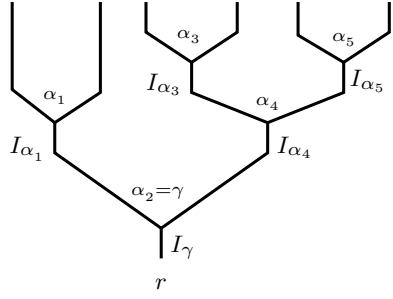}
\end{center}
\caption{Example of a tree topology $T$ for $G=\{\alpha^{}_1,\ldots,\alpha^{}_5\}$, with root $r$ and initial branching point $\gamma=\alpha^{}_2$.
Each (internal and root) edge is identified with a certain segment. Here, we have
$I_{\alpha^{}_1} = L_{<\gamma}$, $I_{\alpha^{}_2} = I_{\gamma} = L$, 
$I_{\alpha^{}_3} = \{\gamma  + 1,\ldots,\alpha^{}_3,\ldots,\alpha^{}_4 -1\}$,
$I_{\alpha_4} = L_{>\gamma}$ and $I_{\alpha_5} = L^{}_{>\alpha^{}_4}$. }
\label{fig:treetopology}
\end{figure}
Next, we  define {\em subtrees}.

\begin{definition}[Subtrees and subtree decomposition]\label{defin:subtrees}
Consider $T=(G,m)$ with $\varnothing\neq G\subseteq L$. Then,
for any $\gamma \in H\subseteq G$ and $\alpha\in G$,
a subtree of $T$ is  defined via 
$T^{}_{\alpha}(H)=(G_{\alpha}(H), m|^{}_{G_{\alpha}(H)} )$,
where 
\[
  G^{}_{\alpha}(H) := \{\beta\in G\arrowvert\alpha\preccurlyeq\beta \, \text{ and }
                h\nprec\beta \quad \forall \, h\in H \, \text{ with } \alpha\prec h\} \, ,
\]
and $m|^{}_{G_{\alpha}(H)}$ is the restriction of $m$ to $G_{\alpha}^{}(H)$.
Specifically, we set $m|^{}_{G_{\alpha}(H)}(\alpha) =:r_{\alpha}$ 
for the initial branching point of the 
respective subtree  (so that
 $r=r_{\gamma}$ for consistency).  
The collection $\{T^{}_{\beta}(H)\}_{\beta \in H}$
describes a {\em decomposition} of $T$ into subtrees, where
$T^{}_{\beta}(H)$ 
has initial branching point $\beta$ and internal nodes $G^{}_{\beta}(H)$.
\end{definition}

Intuitively, the decomposition is obtained by `cutting the tree
below each element of $H$'. The tree then disintegrates into the
subtrees $\{T^{}_{\beta}(H)\}_{\beta \in H}$, 
and each element of $H$ appears as the initial branching point of
one of the subtrees;  Figure~\ref{fig:subdecompo} provides an example.
The $T^{}_{\alpha}(H), \alpha \in G \setminus H$, are, in turn, subtrees of these
subtrees; they will also be required in what follows.

Obviously, $G^{}_{\alpha}(H)$ depends on the topology $T$ (via the partial order),
but again  we omit this for economy of notation. 
Note that the subtrees inherit the segments from the original tree,
i.e.,  $I^{}_{\beta}(T^{}_{\alpha}(H)) = I^{}_{\beta}(T) = I^{}_{\beta}$.
Let us mention that  similar subtree decompositions  appear in the
context of molecular phylogeny, for example, \emph{Tuffley's poset}
\citep{GillSteel}.

\begin{figure}[h]
\begin{center}
\includegraphics[scale=1]{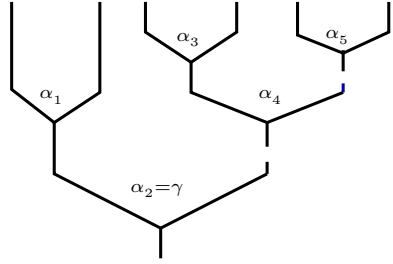}
\end{center}
\caption{Decomposition of a tree topology into three subtrees via
 $H = \{\gamma,\alpha^{}_4,\alpha^{}_5\}$.The subtrees are labelled
with their initial branching points, so the decomposition consists of
$T_{\gamma}(H)$, $T_{\alpha_4}(H)$, and $T_{\alpha_5}(H)$, with node sets
 $G_{\gamma}(H) = \{\alpha^{}_1,\gamma\}$, $G_{\alpha_4}(H) = \{\alpha^{}_3,\alpha^{}_4\}$,
and $G_{\alpha_5}(H) = \{\alpha^{}_5\}$, respectively.}
\label{fig:subdecompo}
\end{figure}

\subsection{ART probabilities and explicit solution of the segmentation process}
Let us now assign probabilities to tree topologies.
To this end, consider the {\em augmented segmentation process}
$\{\widetilde{F}^{}_{\tau}\}_{\tau \in \NN_0}$
with values in the set of all possible tree topologies $T=(G,m)$
(rather than the sets $G$ alone); $\widetilde{F}^{}_{\tau} =(G,m)$ means
that $F^{}_{\tau} =G$, and the segmentation events have occured
according to the partial order implied by $m$ 
(as in \eqref{treeassum} of Example \ref{ex:a13}).
We will abbreviate $\mathbb{P}(\widetilde{F}^{}_{\tau}=T)$
as $\mathbb{P}_{\tau}(T)$.
Let us now state the central result for these tree probabilities.
\begin{theorem}[ART probabilities]\label{thm:treeprob}
Under the segmentation process,
the probability for the tree topology $T=(G,m)$ at time $\tau$ is given by
 $\mathbb{P}^{}_{\tau}(T) = (\lambda^{(L)}_{\varnothing})^{\tau}$ for $G=\varnothing$,
and, for $\varnothing \neq G \subseteq L$ and initial branching point $\gamma \in G$, by
\begin{equation}\label{probtree}
  \mathbb{P}_{\tau}(T) =   \sum_{\gamma \in H \subseteq G} (-1)^{|H|-1} \,
              \big [(\lambda^{(L)}_{G_{\gamma}(H)})^{\tau} - (\lambda^{(L)}_{\varnothing})^{\tau} \big ] \,
               f(T,H).
\end{equation}
Here, for $H\subseteq G$, 
\begin{equation}\label{fth}
   f(T,H) := \prod_{\alpha \in G} g(T^{}_{\alpha}(H))
\end{equation}
with
\begin{equation}\label{gfunc}
 g(T^{}_{\alpha}(H)) := 
  \frac{\varrho^{}_{\alpha}}{\lambda_{G_{\alpha}(H)}^{(I_{\alpha})} 
  - \lambda_{\varnothing}^{(I_{\alpha})}} \quad\text{ for $\alpha\in G$, $H\subseteq G$} \, .
\end{equation}
\end{theorem}

\begin{remark}
\label{rem:alltildeL}
A few remarks are in order:
\begin{enumerate}
\item Note that the dependence on $\tau$ is solely due to the 
term in square brackets in \eqref{probtree}, whereas $f$ is
independent of time.
\item The same term  implies that 
$\PP_0((G,m))=0$ for all $G \neq \varnothing$ and all tree-consistent
mappings $m$. 
\item The $g(T_{\alpha}(H))$ are well-defined and strictly positive since 
$G_{\alpha}(H) \neq \varnothing$
implies that $\lambda_{G_{\alpha}(H)} > \lambda_{\varnothing}$ (cf.\ Lemma~5 in
\citealt{discretereco}). 
\item Eq.~\eqref{probtree} implies a sum over all
possible subtree decompositions of $T$;
for every given decomposition, the subtree containing $\gamma$ 
plays a special role.
\item  We have implied here that the underlying set of links is $L$, i.e.,
$\mathbb{P}_{\tau}(T) = \mathbb{P}(\widetilde{F}^{(L)}_{\tau}=T)$.
However, due to Prop.~\ref{prop:marg}, the result carries over if $L$
is replaced by any contiguous $\tilde L \subseteq L$ that contains $G$. 
\end{enumerate}
\end{remark}

Before we embark on the proof, let us briefly comment on the general strategy.
Like every Markov chain, the segmentation process 
can be viewed in  forward or  in  backward direction
(with respect to the time increment on the $\tau$ time scale):  
If the increment is at the end of the time interval, then
the corresponding ultrametric tree grows at its top
(i.e. the external branches are extended
or split up); otherwise it grows at its
base (i.e. the root branch 
is extended or the two corresponding subtrees coalesce).
Where the original formulation (Definition~\ref{defin:segmentation})
is in the bottom-up direction,  the advantage of the top-down approach
is that one only has to deal with two objects in every step, 
namely the left and the right subtrees
that emerge via the first segmentation event on the  $\tau$
timescale
(and that are joined when looking back), 
instead of a possibly large number of smaller segments at the top.
This point of view has already been used in the proof of 
Theorem~\ref{thm:segisa} and will again serve in the following proof.
\begin{proof}[of Theorem $\ref{thm:treeprob}$.]
We will prove the claim via induction in the top-down direction by progressively
merging pairs of subtrees.
To do so, we first need some properties 
related to the corresponding tree decomposition, see Figure~\ref{fig:treezusamm}.
Consider a tree topology $T=(G,m)$, with 
$\varnothing\neq G\subseteq  L$.
The initial branching point of $T$ is  $\gamma\in G$ as before.
If $\gamma$ has two (internal) offspring nodes, these are denoted by $\gamma^{'}\in G_{<\gamma}\subseteq  L_{<\gamma}$
and $\gamma^{''}\in G_{>\gamma}\subseteq  L_{>\gamma}$. We then define the left
subtree $T^{'}$ of $T$ as 
$T^{'} =  (G_{<\gamma},m|^{}_{G_{<\gamma}})$
and analogously the right subtree $T^{''}$ as 
$T^{''} =  (G_{>\gamma},m|^{}_{G_{>\gamma}}) $,
 both obviously with fewer nodes than $T$.
For convenience, we will denote $ L_{<\gamma}$ ($ L_{>\gamma}$) by $L^{'}$ ($L^{''}$) and
the respective nodes by $G^{'} = G_{<\gamma}$ ($G^{''}= G_{>\gamma}$).
If $\gamma$ has no or only one offspring node, we take the
empty tree (where nothing happens) as left and/or right subtree.
In any case, the offspring nodes of  $\gamma$ are specified through
the preimage of $\gamma$ under the function $m$, i.e.
$m^{-1}(\gamma)\in\{ \{\varnothing\}, \{\gamma^{'}\},\{\gamma^{''}\}, \{\gamma^{'},\gamma^{''}\}\}$.
$T$ is then  obtained by joining these subtrees together at $\gamma$.
In terms of the segmentation process, this corresponds
to the very first cut (of $L$, at link $\gamma$).
Since this may happen at any time $j\in\{1,\ldots,\tau\}$,
(i.e.  the root branch lasts  for $i=j-1\in\{0,\ldots, \tau-1\}$ times while
 $T^{'}$ and $T^{''}$ apply for the remaining $\tau\!-\!1\!-\!i$ time steps),
it is clear that
\begin{equation}\label{proofansatz}
 \mathbb{P}^{}_{\tau}(T) = \sum_{i=0}^{\tau-1}\lambda_{\varnothing}^i\varrho^{}_{\gamma}
   \mathbb{P}^{}_{\tau-1-i}(T^{'})\mathbb{P}^{}_{\tau-1-i}(T^{''}). 
\end{equation}
\begin{figure}
\begin{center}
 \includegraphics{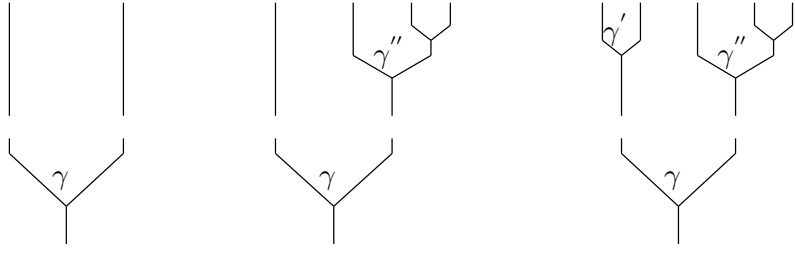}
\end{center}
\caption{Joining together the left and the right subtrees 
($T'$ and $T''$) at 
the initial branching point $\gamma$. Three different cases arise, depending on whether one or both
of the subtrees are  empty trees.}
\label{fig:treezusamm}
\end{figure}
Before we can evaluate \eqref{proofansatz}, we need three preparatory results.\\
\subsubsection*{{\rm (A)}  Product rule for the $\lambda$'s}
For all $L= L^{'}\cup L^{''}\cup \{\gamma\}$ and all $G^{'}\subseteq L^{'}$, $G^{''}\subseteq L^{''}$,
we have 
\begin{equation}\label{prodlam}
 \lambda^{(L^{'})}_{G^{'}} \cdot\lambda^{(L^{''})}_{G^{''}} = \lambda^{( L)}_{G^{'}\cup G^{''}\cup\{\gamma\}} \, ,
\end{equation}
which follows immediately from \eqref{lamkoeff}.
\subsubsection*{{\rm (B)} Product rule for $f$}
For all $H\subseteq G$ 
with $H^{'}:= G^{'}\cap H$ and
$H^{''}:= G^{''}\cap H$, one has
\begin{equation}\label{GtoGprime}
 \begin{split}
   G^{}_{\alpha}(H) &= G^{}_{\alpha}(H^{'}) = G^{'}_{\alpha}(H^{'}) \quad\text{ for all $\alpha\in G^{'}$ and} \\
   G^{}_{\alpha}(H) &= G^{}_{\alpha}(H^{''}) = G^{''}_{\alpha}(H^{''}) \quad\text{ for all $\alpha\in G^{''}$} \, ,
 \end{split}
\end{equation}
so that consequently
\begin{equation}
 \begin{split}
  T^{}_{\alpha}(H) &= T^{'}_{\alpha}(H^{'}) \quad\text{ for all $\alpha\in G^{'}$ and} \\
   T^{}_{\alpha}(H) &= T^{''}_{\alpha}(H^{''}) \quad\text{ for all $\alpha\in G^{''}$} \, .
 \end{split}
\end{equation}
This then leads to the following product rule
for the function $f$ from \eqref{fth}:
\begin{equation}\label{prodfth}
\begin{split}
 f(T,H) &= g(T^{}_{\gamma}(H))\cdot \prod_{\alpha\in G^{'}}g(T^{}_{\alpha}(H))
   \prod_{\beta\in G^{''}}g(T^{}_{\beta}(H))  \\
   &= g(T^{}_{\gamma}(H))\cdot \prod_{\alpha\in G^{'}}g(T^{'}_{\alpha}(H^{'}))
   \prod_{\beta\in G^{''}}g(T^{''}_{\beta}(H^{''})) 
= g(T^{}_{\gamma}(H)) f(T^{'},H^{'})f(T^{''},H^{''})\, .
\end{split}
\end{equation}
Note that in case $G^{'}=\varnothing$ or $G^{''}=\varnothing$, the corresponding
empty product is  $1$ as usual.
\subsubsection*{{\rm (C)} Assembly of initial branching point and subtrees}
As an immediate consequence of  Definition~\ref{defin:subtrees}, one obtains
for all $C\subseteq m^{-1}(\gamma) \subseteq H\subseteq G$:
\begin{equation}\label{mengengym}
\{\gamma \}\cup  \bigcup_{\eta\in C} G^{}_{\eta}(H) = 
G^{}_{\gamma} (H\setminus C ) = G^{}_{\gamma} ((H\cup \{\gamma\})\setminus C )\, .
\end{equation}
(In fact, this relationship is not restricted to $\gamma$ but also holds 
for arbitrary $\alpha\in G$,
in an analogous way; but this will not be used in what follows.)

\subsubsection*{{\rm (D)} Subtree summation}
Recall  the classic identity
\begin{equation}\label{lemma:geom}
 \sum_{i=0}^n a^i b^{n-i}  
 = \begin{cases} \frac{b^{n+1} - a^{n+1}}{b-a} \,, & a \neq b, \\
                 (n+1) a^n \,, & a = b,
\end{cases}
\end{equation}
where the second case is also the l'H\^{o}pital limit of the first.
Together with \eqref{gfunc}, this implies
\begin{equation} \label{lemmconseq}
\begin{split}
 \varrho^{}_{\alpha}\sum_{i=0}^{\tau-1}
   (\lambda^{(I_{\alpha})}_{\varnothing})^i(\lambda^{(I_{\alpha})}_{G^{}_{\alpha}(H)})^{\tau-i-1}
   &= \frac{\varrho_{\alpha}}{\lambda^{(I_{\alpha})}_{G^{}_{\alpha}(H)} - \lambda^{(I_{\alpha})}_{\varnothing}}
\cdot \big ((\lambda^{(I_{\alpha})}_{G^{}_{\alpha}(H)})^{\tau} - (\lambda^{(I_{\alpha})}_{\varnothing})^{\tau} \big ) \\
  &= g\big (T_{\alpha}(H) \big ) \big ((\lambda^{(I_{\alpha})}_{G^{}_{\alpha}(H)})^{\tau} - (\lambda^{(I_{\alpha})}_{\varnothing})^{\tau} \big )
\end{split}
\end{equation}
for all $T=(G,m)$, $\varnothing\neq G\subseteq  L$, $\alpha\in G$ and $H\subseteq  L$.

\bigskip

We now continue the proof of the theorem and proceed
via induction over $\vert G\vert$.
For $G=\varnothing$, the claim holds trivially, that is, 
$\PP_{\tau}(T) = \lambda_{\varnothing}^{\tau}$  for all
$\tau \geqslant 0$. 
For $G\subseteq L$ with $\vert G\vert =1$, i.e. $G=\{\gamma\}$,
both the left and the right subtrees are empty trees
(see the left case in Figure~\ref{fig:treezusamm}).
Using first \eqref{proofansatz}, then the result for $G^{'} = G^{''} =\varnothing$ on $L'$ and $L''$, respectively (see Remark~\ref{rem:alltildeL}~(5)),
then \eqref{prodlam}, and finally \eqref{lemmconseq}, we obtain
\[
\begin{split}
 \mathbb{P}^{}_{\tau}(T) &= 
      \varrho_{\gamma} \sum_{i=0}^{\tau-1}\lambda^i_{\varnothing} 
  \mathbb{P}^{}_{\tau-1-i}(T') \mathbb{P}^{}_{\tau-1-i}(T'') 
  = \varrho_{\gamma} \sum_{i=0}^{\tau-1}\lambda^i_{\varnothing} 
           (\lambda_{\varnothing}^{(L^{'})})^{\tau-1-i}(\lambda_{\varnothing}^{(L^{''})})^{\tau-1-i} \\
 &= \varrho_{\gamma} \sum_{i=0}^{\tau-1}\lambda^i_{\varnothing}
           \lambda_{\gamma}^{\tau-1-i}
   = g\big (T^{}_{\gamma}(\{\gamma\})\big )
          (\lambda_{\gamma}^{\tau} - \lambda_{\varnothing}^{\tau}) \,.
\end{split}
\]
We now assume the claim to hold for all tree topologies 
$T=(G,m)$ for all  $G\subseteq L$
with $\vert G\vert\leqslant k$ for some $k\geqslant 1$; by Remark \ref{rem:alltildeL},
it then holds likewise with $L$ replaced by a contiguous 
$\tilde L \subseteq L$ as
long as $G \subseteq \tilde L$. 
We next turn to $G=\{\alpha^{}_1,\ldots,\alpha^{}_{k+1}\}\subseteq 
L$
and fixed $T=(G,m)$; recall our convention $\alpha_1 < \alpha_2 < \ldots <
\alpha_{k+1}$.
In the following, we will always write $A\cup \alpha := A\cup\{\alpha\}$
and $A\setminus\alpha:=A\setminus\{\alpha\}$ for $A\subseteq  L$ and $\alpha\in L$.
We have to distinguish the two remaining cases in Figure~\ref{fig:treezusamm}
(middle and right):
\subsubsection*{Case $\gamma=\alpha^{}_1$}
Then $T^{'}$
is an empty tree while $T^{''}$ has initial branching point $\gamma^{''}\in G^{''}$, i.e. $m^{-1}(\gamma) =\{\gamma^{''}\}$.
We then find
\[
 \begin{split}
  \mathbb{P}^{}_{\tau}(T) &= \varrho^{}_{\gamma} \sum_{i=0}^{\tau-1}\lambda_{\varnothing}^i
   \mathbb{P}^{}_{\tau-1-i}(T^{'})\mathbb{P}^{}_{\tau-1-i}(T^{''}) \\
&=\varrho^{}_{\gamma}\sum_{i=0}^{\tau-1}\lambda_{\varnothing}^i
      (\lambda^{(L^{'})}_{\varnothing})^{\tau-1-i} \!\!\!\!\!\!\!
  \sum_{\gamma^{''}\in H^{''}\subseteq G^{''}} \!\!\!(-1)^{\vert H^{''}\vert -1}
\Bigl((\lambda^{(L^{''})}_{G^{''}_{\gamma^{''}}(H^{''})})^{\tau-1-i} - 
         (\lambda^{(L^{''})}_{\varnothing})^{\tau-1-i}\Bigr) f(T^{''},H^{''}) \\
&=\varrho^{}_{\gamma}\sum_{i=0}^{\tau-1}\lambda_{\varnothing}^i
  \sum_{\gamma^{''}\in H^{''}\subseteq G^{''}} \!\!\!(-1)^{\vert H^{''}\vert -1}
\big(\lambda_{\gamma \cup G^{''}_{\gamma^{''}}(H^{''})}^{\tau-1-i} - 
         \lambda_{\gamma}^{\tau-1-i}\big) f(T^{''},H^{''}) \\
&= \varrho^{}_{\gamma} \sum_{i=0}^{\tau-1}\lambda_{\varnothing}^i
   \sum_{\gamma^{''}\in H^{''}\subseteq G^{''}} \!\!\!(-1)^{\vert H^{''}\vert -1}
\big(\lambda^{\tau-1-i}_{G^{}_{\gamma^{}}((H^{''}\cup \gamma)\setminus \gamma^{''})} - 
         \lambda^{\tau-1-i}_{G^{}_{\gamma^{}}(H^{''}\cup \gamma)}\big) f(T^{''},H^{''}) \\
 &= \sum_{\gamma^{''}\in H^{''}\subseteq G^{''}} (-1)^{\vert H^{''}\vert -1}
      g(T_{\gamma}((H^{''}\cup\gamma)\setminus \gamma^{''}))
    \bigl(\lambda_{G^{}_{\gamma^{}}((H^{''}\cup \gamma)\setminus \gamma^{''})}^{\tau} - 
         \lambda_{\varnothing}^{\tau}\bigr) f(T^{''},H^{''}) \\
  &\phantom{=} - \sum_{\gamma^{''}\in H^{''}\subseteq G^{''}} 
                 (-1)^{\vert H^{''}\vert -1} g(T_{\gamma}(H^{''}\cup\gamma))
\bigl(\lambda_{G^{}_{\gamma^{}}(H^{''}\cup \gamma)}^{\tau} - \lambda_{\varnothing}^{\tau}\bigr)
     f(T^{''},H^{''}) \\
  &= \sum_{\gamma^{''}\in H^{''}\subseteq G^{''}} (-1)^{\vert H^{''}\vert -1}
        \bigl(\lambda^{\tau}_{G^{}_{\gamma^{}}((H^{''}\cup \gamma)\setminus \gamma^{''})} - 
         \lambda_{\varnothing}^{\tau}\bigr) f(T,(H^{''}\cup \gamma)\setminus \gamma^{''}) \\   
   &\phantom{=} - \sum_{\gamma^{''}\in H^{''}\subseteq G^{''}} (-1)^{\vert H^{''}\vert -1}
   \bigl (\lambda_{G^{}_{\gamma^{}}(H^{''}\cup \gamma)}^{\tau} - 
\lambda_{\varnothing}^{\tau}\bigr) f(T,H^{''}\cup\gamma) \\
  &= \sum_{\gamma\in H\subseteq G\setminus \gamma^{''}} (-1)^{\vert H\vert -1} 
    \bigl(\lambda_{G^{}_{\gamma^{}}(H)}^{\tau} - 
         \lambda_{\varnothing}^{\tau}\bigr)f(T,H) - \!\!\!\!\!
    \sum_{\{\gamma,\gamma^{''}\}\subseteq H\subseteq G} (-1)^{\vert H\vert} \bigl(\lambda_{G^{}_{\gamma^{}}(H)}^{\tau} - 
         \lambda_{\varnothing}^{\tau}\bigr)f(T,H) \\
  &= \sum_{\gamma\in H\subseteq G} (-1)^{\vert H\vert - 1}
     \bigl(\lambda_{G^{}_{\gamma^{}}(H)}^{\tau} - 
         \lambda_{\varnothing}^{\tau}\bigr)f(T,H) \, .
 \end{split}
\]
In the first step, we have  used \eqref{proofansatz}, in the
second  the induction hypothesis (applied to $T''$ on $L''$),
in the third the product structure of the $\lambda$'s \eqref{prodlam},
in the fourth \eqref{GtoGprime} (read in the backward direction) and 
\eqref{mengengym} (applied to $H=H''$ with $C=\{\gamma''\}$ and
$C=\varnothing$, respectively).  
In the fifth step, we have invoked \eqref{lemmconseq} (separately 
on each term in parentheses), in the sixth step we have
used  \eqref{prodfth} 
with $H=(H^{''}\cup\gamma)\setminus \gamma^{''}$ and
$H=H^{''}\cup\gamma$, respectively, and 
finally we have changed the summation variable
(where we set $G=\gamma\cup G^{''}$).

\subsubsection*{Case $\gamma\in\{\alpha^{}_2,\ldots, \alpha^{}_{k}\}$}
Now we have to consider
the left subtree $T^{'}$ with initial branching point $\gamma^{'}$ and the right subtree $T^{''}$ with 
initial branching point $\gamma^{''}$, i.e. $m^{-1}(\gamma) = \{\gamma^{'},\gamma^{''}\}$. 
Proceeding in analogy with the previous case, we obtain
\[
 \begin{split}
  \mathbb{P}^{}_{\tau}(T) &= \sum_{i=0}^{\tau-1}\lambda_{\varnothing}^i\varrho^{}_{\gamma}
   \mathbb{P}^{}_{\tau-1-i}(T^{'})\mathbb{P}^{}_{\tau-1-i}(T^{''}) \\
&=\sum_{i=0}^{\tau-1}\lambda_{\varnothing}^i\varrho^{}_{\gamma}
  \sum_{\gamma^{'}\in H^{'}\subseteq G^{'}} \!\!\!(-1)^{\vert H^{'}\vert -1}
\Bigl((\lambda^{(L^{'})}_{G^{'}_{\gamma^{'}}(H^{'})})^{\tau-1-i} - 
         (\lambda^{(L^{'})}_{\varnothing})^{\tau-1-i}\Bigr) f(T^{'},H^{'}) \\
&\phantom{=}\times\sum_{\gamma^{''}\in H^{''}\subseteq G^{''}} \!\!\!(-1)^{\vert H^{''}\vert -1}
\Bigl((\lambda^{(L^{''})}_{G^{''}_{\gamma^{''}}(H^{''})})^{\tau-1-i} - 
         (\lambda^{(L^{''})}_{\varnothing})^{\tau-1-i}\Bigr) f(T^{''},H^{''}) \\
&= \sum_{i=0}^{\tau-1}\lambda_{\varnothing}^i\varrho^{}_{\gamma}
   \sum_{\gamma^{'}\in H^{'}\subseteq G^{'}}
   \sum_{\gamma^{''}\in H^{''}\subseteq G^{''}}  (-1)^{\vert H^{'}\vert + \vert H^{''}\vert}
  \Bigl( \lambda^{\tau-1-i}_{G^{}_{\gamma}((H^{'}\cup H^{''}\cup \gamma)\setminus \{\gamma^{'}, \gamma^{''}\})} \\
    &\phantom{=} - \lambda^{\tau-1-i}_{G^{}_{\gamma}((H^{'}\cup H^{''}\cup\gamma)\setminus \gamma^{''})}
      - \lambda^{\tau-1-i}_{G^{}_{\gamma}((H^{'}\cup H^{''}\cup\gamma)\setminus \gamma^{'})}
   + \lambda^{\tau-1-i}_{G^{}_{\gamma}(H^{'}\cup H^{''}\cup\gamma)}\Bigr)f(T^{'},H^{'})f(T^{''},H^{''}) \\
 &=  \sum_{\gamma^{'}\in H^{'}\subseteq G^{'}}
   \sum_{\gamma^{''}\in H^{''}\subseteq G^{''}}  (-1)^{\vert H^{'}\vert + \vert H^{''}\vert}
  \Bigl(g(T^{}_{\gamma}((H^{'}\cup H^{''}\cup\gamma)\setminus\{\gamma^{'},\gamma^{''}\}))
(\lambda_{G^{}_{\gamma}((H^{'}\cup H^{''}\cup \gamma)\setminus \{\gamma^{'}, \gamma^{''}\})}^{\tau}
   - \lambda^{\tau}_{\varnothing}) \\
  &\phantom{=} - g(T^{}_{\gamma}((H^{'}\cup H^{''}\cup\gamma)\setminus\{\gamma^{''}\}))
(\lambda_{G^{}_{\gamma}((H^{'}\cup H^{''}\cup \gamma)\setminus \{\gamma^{''}\})}^{\tau})
   - \lambda_{\varnothing}^{\tau}) \\
   &\phantom{=} - g(T^{}_{\gamma}((H^{'}\cup H^{''}\cup\gamma)\setminus\{\gamma^{'}\}))
(\lambda_{G^{}_{\gamma}((H^{'}\cup H^{''}\cup \gamma)\setminus \{\gamma^{'}\})}^{\tau}
   - \lambda_{\varnothing}^{\tau})  \\
  &\phantom{=} + g(T^{}_{\gamma}(H^{'}\cup H^{''}\cup\gamma))
(\lambda_{G^{}_{\gamma}(H^{'}\cup H^{''}\cup \gamma)}^{\tau}
   - \lambda^{\tau}_{\varnothing}) \Bigr) f(T^{'},H^{'})f(T^{''},H^{''}) \\
&= \!\!\!\!\!\!\sum_{\gamma\in H\subseteq G\setminus\{\gamma^{'},\gamma^{''}\}} \!\!\!\!\!
       (-1)^{\vert H\vert -1}(\lambda_{G^{}_{\gamma}(H)}^{\tau} - \lambda_{\varnothing}^{\tau})f(T,H)
  - \!\!\!\!\!\!\!\!\sum_{\{\gamma, \gamma^{'}\}\subseteq H\subseteq G\setminus\{\gamma^{''}\}}\!\!\!\!\!
       (-1)^{\vert H\vert}(\lambda_{G^{}_{\gamma}(H)}^{\tau} - 
\lambda_{\varnothing}^{\tau})f(T,H) \\
  &\phantom{=} - \!\!\!\!\!\!\!\!\sum_{\{\gamma, \gamma^{''}\}\subseteq H\subseteq G\setminus\{\gamma^{'}\}}\!\!\!\!\!
       (-1)^{\vert H\vert}(\lambda_{G^{}_{\gamma}(H)})^{\tau} - 
\lambda_{\varnothing}^{\tau})f(T,H)
   + \!\!\!\!\!\!\!\!\sum_{\{\gamma, \gamma^{'},\gamma^{''}\}\subseteq H\subseteq G}\!\!\!\!\!
       (-1)^{\vert H\vert -1}(\lambda_{G^{}_{\gamma}(H)}^{\tau} - 
\lambda_{\varnothing}^{\tau})f(T,H) \\
  &= \sum_{\gamma\in H\subseteq G} (-1)^{\vert H\vert - 1}
     \bigl(\lambda_{G^{}_{\gamma^{}}(H)}^{\tau} - 
         \lambda_{\varnothing}^{\tau}\bigr)f(T,H) \, .
 \end{split}
\]
The remaining case $\gamma=\alpha_{k+1}$ is analogous to $\gamma=\alpha_1$.
\qed
\end{proof}
Now that we have an explicit formula for the probability of our
tree topologies, 
let us further comment on its structure.
\begin{remark}
Using the subtree decomposition of Def.~\ref{defin:subtrees},
we can  state \eqref{fth} alternatively as
\begin{equation}\label{fthnew}
 f(T,H) = \prod_{\alpha\in H}\prod_{\beta\in G^{}_{\alpha}(H)} g(T^{}_{\beta}(H)) \,,
\end{equation}
which implies some kind of independence across subtrees.
\end{remark}

Returning to the segmentation process, we may conclude
that 
\[
\mathbb{P}(F_{\tau} = G) = \sum_m \PP(\tilde F_{\tau}=(G,m)),
\] 
where the sum is
over all mappings consistent with a full binary tree.
The final result then follows directly from Theorem~\ref{thm:segisa}.
\begin{corollary}[Solution of recombination equation via ARTs]\label{coro:exsol}
 The discrete-time recombination equation \eqref{rekooper} has the solution
\[
p^{}_t     =  \sum_{G\subseteq L} a^{}_G ( t ) R^{}_G 
( p^{} _0 ),
\] 
where
\begin{equation}\label{exsol}
 a^{}_G(t) = 
\sum_{m}\PP(\tilde F_t=(G,m))  
\end{equation}
for all $G\subseteq L$, where the sum is over all tree-consistent $m$,
with
$\PP(\tilde F_t=(G,m))=\mathbb{P}_t(T)$ as given in Theorem~{\rm\ref{thm:treeprob}}. \qed
\end{corollary}


\section{Discussion}
\label{sec:discussion}
The piece of research presented here has solved the
long-standing deterministic single-crossover dynamics forward in time
by considering the corresponding stochastic process (the Wright-Fisher
model with recombination) backward in time, via looking at single individuals
and tracing back their ancestries. We will first compare the result
with the previous recursive approaches 
and then turn to the connection with the ancestral 
recombination graph (ARG; the usual approach
to recombination in \emph{finite} populations, see 
\citealt[Chap.~7.2]{Wakeley} and
\citealt[Chap.~3.4]{Durrett}).

Evaluating the coefficients $a^{}_G(t)$  via the traditional recursive
approaches is an algebraic strategy, which relies on
linearisation and diagonalisation of the underlying (forward, deterministic)
dynamical system. In contrast, the ART approach presented here starts from
a summation over all paths of the (backward, stochastic) process that
give rise to a given set of segments after $t$ generations; this is
reflected in the sum over all ultrametric trees, as in
\eqref{ex13eq}. This formula  contains
sums over all tree toplogies \emph{and} over all combinations of branch
lengths, which is not useful in itself, in particular for large $t$.
The simplification obtained here consists in carrying out the 
summation over the branch lengths, so that one is left with the tree topologies
only. This is particularly useful for large $t$ and small recombination
probabilities, since long branches (where nothing happens) are `contracted'
in this way.

Both the recursive and the ART solution are of similar computational complexity.
Wheter  $a^{}_G(t)$ is evaluated via ARTs
(Corollary~\ref{coro:exsol} and Theorem~\ref{thm:treeprob}) or by solving 
 the recursions (e.g., \citealt[Eqs.~(43)--(45)]{discretereco} or the
related ones  of \citealt{Kevin1,Kevin2}), the effort grows exponentially
with $n:=\lvert G \rvert$.  In  the recursions,
the complexity comes from  multiple sums over nested sets of subsets of $G$.
In the case of the ARTs, it is
due to the  summation over all possible tree
topologies with internal node set $G$; their number grows exponentially
with $n$. To be precise, there are $C(n)$ such topologies
(\citealt[Chap.~3.4]{GrossYellen}; \citealt[Ex.~6.19.d]{Stanley}),
where $C(n)$ is the $n$'th Catalan number. After all,
clever  algorithms are available for the generation and enumeration of
these trees \citep{Gupta,Proskur,Zaks}. 

The ART formula is therefore not superior in computational terms.
However, it bears the great advantage to  relate to objects with 
an immediate meaning in terms of the
underlying process. After all, the probabilities for the
tree topologies may lend themselves to future use if, for example,
one is interested in the distribution of tree shape(s), or if
mutation is included in the model (that is, superimposed on the trees).
This is in contrast with the
manifestly non-intuitive recursions, for which we are not aware of
an interpretation in terms of the underlying process.

The backward approach that gives rise to the ARTs differs from the ARG
in two ways. First, we let $N$ tend to infinity without
rescaling any parameters; that is, recombination probabilities
remain constant when $N \to \infty$. This corresponds to the assumption
that recombination is so strong (that is, loci are so far apart)
that the majority of the recombination events takes place before
coalescence sets in. In contrast, the ARG assumes weak recombination
(in that recombination parameters scale inversely with population size),
so that recombination and coalescence take place on the same time scale.
Second, we focus on the ancestry of single individuals rather than of
samples, which further simplifies matters. As a reward, one obtains
semi-explicit answers for all quantities of interest here.

The differences between the scopes  of the two approaches 
can be  illustrated nicely
in the context of the paper by  \cite{WiufHein}, who analyse 
the ancestry of the genetic material of an entire chromosome from
a single individual. 
More precisely, they investigate the partitioning process
(a  close relative of our $\{\Sigma_{\tau}\}_{\tau \geqslant 0}$) 
\emph{at stationarity}, i.e., for $\tau \to \infty$, 
in the diffusion limit. They employ
approximations and simulations and focus on the concrete example of
the human chromosome 1, with realistic estimates of the 
recombination parameters and the
effective population size. 
They do find large contributions from unordered partitions,
in the sense of the frequent occurrence of
segments of ancestral material interspersed with
nonancestral (`trapped') material between them.
This is a consequence of the
large number of genetic ancestors of the chromosome (estimated at 6800)
relative to the effective population size ($N_e=20000$).
(For loci at opposite ends of the chromosome, the diffusion limit
is, however, not expected to yield a good approximation, see below.)

In contrast, our approach does not aim at any stationary situation.
Rather, it should provide a faithful picture in situations
where recombination rates and population size are large enough so that
there is a time horizon governed by recombination alone.
The detailed time course of  individual ancestries 
over this time horizon
will then be described by the segmentation process.
Let us note that a well-known method of  simulating the ARG (the so-called
sequential Markov coalescent by 
\citealt{McVean}) also uses an
approximation of the partitioning process by the segmentation
process in that  coalescence of sequences that both carry 
ancestral material are neglected; the authors demonstrate that this
yields a good approximation to the full ARG over a wide range of parameters.

To be more precise, we stipulate that there are, in fact, three
scaling regimes to be considered in the context of recombination:
weak recombination, strong recombination, and free recombination.
We can certainly not delineate them precisely here (and they are
also expected to overlap). However, with the usual assumption 
of a recombination probability of the order of $10^{-8}$ per
generation and  base pair of a DNA sequence \citep{Kauppi},
it is clear that
sites at a distance of the order of $10^3$--$10^4$ base pairs
will fall into the weak recombination regime:  
The recombination probability between them, of $10^{-5}$--$10^{-4}$,
is of the same order as the
coalescence probability of $1/N$ (between any pair of branches per generation).
In contrast, sites at a distance of $10^8$ base pairs (like the opposing
ends of a chromosome) will be essentially independent since there
will be, on average, one crossover between them  in every
generation; this is the case of free recombination. Between the
extremes, at a distance of $10^6$ base pairs, say, the recombination
probability of $10^{-2}$ between a pair of sites is well-separated from
the coalescence probability of $1/N$ per pair of branches
for all but the smallest populations;
this should be a case for strong recombination over a time horizon
where the number of branches is not too large. Note that the
`number of branches' that counts here is the number of genetic ancestors
of an individual (that is, those that carry ancestral material),
not the number of genealogical ancestors (that is, all parents, grandparents
and so on). The number of genetic ancestors only increases roughly
linearly with (backward) time $\tau$, whereas an individual has
up to $2^{\tau}$ genealogical ancestors $\tau$ generations back; see the
discussions by \cite{Donelly} and \cite{RalphCoop}. Note also that, in
this parameter regime, the  assumption of at most one crossover
per generation is well justified,
while recombination is still strong relative to coalescence. Last not least,
it is worth mentioning that this is the parameter regime where the deterministic
dynamics yields a valid  description.

The results presented here should also pave the way towards
a solution of the  multiple crossover model. Biologically,
this is relevant when more distant loci are considered. 
Recombination in a given generation will again be described 
by a partition
of $S$ into two parts (corresponding to the two parents), but, 
this time, these partitions
may be  arbitrary (as opposed to the ordered partitions that arise
due to single crossovers). 
But we  expect that the corresponding ancestral recombination
trees  will continue to be binary trees
whose subtrees are conditionally independent, so that the methods
developed here may be generalised to this case involving arbitrary
partitions.

\begin{acknowledgements}
It is our pleasure to thank the `recombination seminar'
(M.~Baake, C.~Huck, T.~Hustedt, P.~Zeiner) for thorough discussion of all stages
of this research, and A.~Wakolbinger for an enlightening conversation.
We are grateful to M.~Baake, K.~Schneider, M.~Esser, and S.~Probst for critically
reading  the manuscript, and to an anonymous referee for valuable
suggestions. This work was supported by Deutsche
Forschungsgemeinschaft (DFG) in the framework of
the Research Training Group `Bioinformatics' at Bielefeld University
(GK 635), and of the Priority Programme `Probabilistic Structures in
Evolution', SPP 1590 (BA1070/12-1 and BA2469/7-1).
\end{acknowledgements}

\bibliographystyle{spr-chicago}      
\bibliography{example}   
\nocite{*}


\end{document}